\documentclass{amsart}

\usepackage{amsmath}
\usepackage{amssymb}

\newtheorem*{MainThm}{Main Theorem}{\bf}{\it}

\newtheorem{theorem}{Theorem}[section]
\newtheorem{lemma}[theorem]{Lemma}
\newtheorem{proposition}[theorem]{Proposition}

\theoremstyle{definition}

\theoremstyle{remark}
\newtheorem{remark}[theorem]{Remark}

\numberwithin{equation}{section}

\newcommand{\R}{\mathbb{R}}
\newcommand{\C}{\mathbb{C}}
\newcommand{\ad}{\mathrm{ad}}

\newcommand{\vol}{\mathrm{vol}}
\newcommand{\rad}{\mathrm{rad}}
\newcommand{\uSO}{\widetilde{\mathrm{SO}}}
\newcommand{\SL}{\mathrm{SL}}
\newcommand{\Kill}{\mathrm{Kill}}
\newcommand{\Iso}{\mathrm{Iso}}
\newcommand{\iso}{\mathfrak{Iso}}
\newcommand{\OO}{\mathcal{O}}
\newcommand{\HH}{\mathcal{H}}
\newcommand{\GG}{\mathcal{G}}
\newcommand{\VV}{\mathcal{V}}
\newcommand{\g}{\mathfrak{g}}
\newcommand{\h}{\mathfrak{h}}
\newcommand{\s}{\mathfrak{s}}
\newcommand{\su}{\mathfrak{su}}
\newcommand{\symp}{\mathfrak{sp}}

\newcommand{\Z}{\mathbb{Z}}
\newcommand{\SO}{\mathrm{SO}}
\newcommand{\so}{\mathfrak{so}}
\newcommand{\Tr}{\mathrm{tr}}
\newcommand{\gsl}{\mathfrak{sl}}

\usepackage{appendix}

\begin{document}

\title{On low-dimensional manifolds with isometric
  $\mathrm{SO}_0(p,q)$-actions}

\author{Gestur Olafsson}
\address{Department of Mathematics, 303 Lockett Hall, Louisiana State
  University, Baton Rouge, LA, 70803, USA}
\email{olafsson@math.lsu.edu}
\thanks{The first author was supported by DMS grants 0402068 and
  0801010}

\author{Raul Quiroga-Barranco}
\address{Centro de Investigaci\'on en Matem\'aticas, Apartado Postal
  402, Guanajuato, Guanajuato, 36250, Mexico}
\email{quiroga@cimat.mx}
\thanks{The second author was supported by Conacyt, Concyteg and SNI}

\subjclass{57S20, 53C50, 53C24}
\keywords{Pseudo-Riemannian manifolds, simple Lie groups, rigidity
  results}

\begin{abstract}
  Let $G$ be a non-compact simple Lie group with Lie algebra
  $\g$. Denote with $m(\g)$ the dimension of the smallest non-trivial
  $\g$-module with an invariant non-degenerate symmetric bilinear
  form. For an irreducible finite volume pseudo-Riemannian analytic
  manifold $M$ it is observed that $\dim(M) \geq \dim(G) + m(\g)$ when
  $M$ admits an isometric $G$-action with a dense orbit. The Main
  Theorem considers the case $G = \widetilde{\mathrm{SO}}_0(p,q)$
  providing an explicit description of $M$ when the bound is
  achieved. In such case, $M$ is (up to a finite covering) the
  quotient by a lattice of either $\widetilde{\mathrm{SO}}_0(p+1,q)$
  or $\widetilde{\mathrm{SO}}_0(p,q+1)$.
\end{abstract}

\maketitle

\section*{Introduction}
\noindent
Let $G$ be a connected non-compact simple Lie group acting
isometrically on a connected analytic manifold $M$ with a finite
volume pseudo-Riemannian metric. Following Zimmer's program, it has
been shown that such actions are rigid in the sense of having
distinguished properties that restrict the possibilities for $M$ (see
for example \cite{Gromov,Zimmer-Lorentz,Zimmer-rigid}). The general
belief is that any such action, with some additional non-triviality
conditions, must essentially be an algebraic double coset of the form
$K\backslash H /\Gamma$. More precisely, such coset is given by some
Lie group $H$ together with a homomorphism $G \rightarrow H$, a
lattice $\Gamma \subset H$ and a compact subgroup $K \subset H$
centralizing the image of $G$ in $H$. The $G$-action is then given by
the natural left action on $K\backslash H /\Gamma$. We note that when
$H$ is semisimple these $G$-actions are isometric for a metric induced
by the Killing form of the Lie algebra of $H$. Some results have
already been obtained in \cite{Quiroga-Annals,Quiroga-Z} proving that
suitable geometric conditions imply that such $G$-actions are of the
double coset type. We refer to
\cite{Bader-Frances-Melnick,Bader-Nevo,Frances-Melnick} for similar
related results.

In this work we observe, that for $M$ complete and weakly irreducible
with the $G$-action non-transitive but with a dense orbit, the
dimension of $M$ has a bound from below in terms of the representation
theoretic properties of $\g$, the Lie algebra of $G$. More precisely,
it is noted in Proposition~\ref{prop-bound} that in this case we have
\[
\dim(M) \geq \dim(G) + m(\g),
\]
where $m(\g)$ denotes the dimension of the smallest non-trivial
representation of $\g$ that admits an invariant non-degenerate
symmetric bilinear form. Recall that a connected pseudo-Riemannian
manifold is weakly irreducible if the tangent space at some (and hence
any) point has no proper non-degenerate subspaces invariant under the
restricted holonomy group at that point.

For our main result, we consider with further detail the case $G =
\uSO_0(p,q)$, the universal covering group of $\SO_0(p,q)$. For
$\uSO_0(p,q)$-actions with $p,q \geq 1$ and $p+q \geq 4$, the
following result establishes that the lower bound just considered can
be achieved only for double coset models. Note that for a $G$-action
on a manifold $M$ and $X$ in the Lie algebra of $G$ we denote by $X^*$
the vector field on $M$ whose one-parameter group of diffeomorphisms
is given by $(\exp(tX))_t$ through the $G$-action on $M$.

\begin{MainThm}
  Let $M$ be a connected analytic pseudo-Riemannian manifold. Suppose
  that $M$ is complete weakly irreducible, has finite volume and
  admits an analytic and isometric $\uSO_0(p,q)$-action with a dense
  orbit, for some integers $p,q$ such that $p,q \geq 1$ and $n =
  p+q\geq 5$. In this case we have $m(\so (p,q)) = n$. If the
  equality:
  \[
  \dim(M) = \dim(\uSO_0(p,q)) + m(\so (p,q)) = \frac{n(n+1)}{2},
  \]
  holds, then for $H$ either $\uSO_0(p,q+1)$ or $\uSO_0(p+1,q)$ there
  exist:
  \begin{itemize}
  \item[(1)] a lattice $\Gamma \subset H$, and
  \item[(2)] an analytic finite covering map $\varphi : H/\Gamma
    \rightarrow M$,
  \end{itemize}
  such that $\varphi$ is $\uSO_0(p,q)$-equivariant, where the
  $\uSO_0(p,q)$-action on $H/\Gamma$ is induced by some non-trivial
  homomorphism $\uSO_0(p,q) \rightarrow H$. Furthermore, we can
  rescale the metric on $M$ along the $\uSO_0(p,q)$-orbits and their
  normal bundle to assume that $\varphi$ is a local isometry for the
  bi-invariant pseudo-Riemannian metric on $H$ given by the Killing
  form of its Lie algebra. The result holds for the case $(p,q) =
  (3,1)$ as well if we further assume that $X^* \perp Y^*$ on $M$ for
  all $X \in \su(2)$ and $Y \in i\su(2)$ under the identification
  $\so(3,1) \simeq \mathfrak{sl}(2,\C)$.
\end{MainThm}

Note that there is no $\R$-rank restriction, and so this result
applies to the groups $\uSO_0(p,1)$ when $p \geq 3$. Thus, the Main
Theorem provides a rigidity result for $\SO_0(p,1)$-actions.

The proof of the Main Theorem is based on the application of
representation theory to the Killing vector fields centralizing the
$G$-action, where the latter are as found in Gromov-Zimmer's machinery
(see \cite{Gromov,Zimmer-rigid}). With respect to centralizing Killing
fields, our main ingredient is Proposition~\ref{prop-g(x)} as already
found in \cite{GCT,Gromov,Quiroga-Z,Zimmer-rigid} with varying
assumptions on the manifold $M$ acted upon by
$G$. Proposition~\ref{prop-g(x)} ensures the existence of a Lie
algebra $\g(x)$, isomorphic to $\g$, of Killing fields vanishing at a
point $x$ on the universal cover of $M$ and with some additional
properties. The Lie algebra $\g(x)$ provides a $\g$-module structure
to the tangent space $T_xM$ that allows to use representation theory
to control the behavior of the normal to the orbits. Such $\g$-module
structure is then related to the Lie algebra $\HH$ of Killing vector
fields centralizing the $G$-action (see
Lemma~\ref{lemma-ev-homomorphism}), thus again providing control on
$\HH$. By Gromov-Zimmer's machinery, the Lie algebra of $\HH$ has an
open dense orbit in the universal covering space of $M$. Also, the Lie
algebra $\HH$ yields a Lie group action constructed in
Section~\ref{section-Killing-fields}. These elements together allow to
obtain a very detailed description of the structure of $\HH$, which is
provided in Section~\ref{section-structure-centralizer}. Here again,
the application of representation theory is a key element. Finally,
Section~\ref{section-structure-tildeM} completes the proof of the Main
Theorem using the Lie group action induced by $\HH$ and the weak
irreducibility assumption on $M$.
Some needed facts about the Lie algebra $\so (p,q)$ are collected in
Appendix \ref{section-Lie-algebras}. We will use the notation from the
introduction and the Appendix without further comments.

\setcounter{section}{0}

\section{Isometric actions of simple Lie groups and Killing fields}
\label{section-Killing-fields}
\noindent
In this section, we let $G$ be a connected non-compact simple Lie group
with Lie algebra $\g$ and $M$ a connected finite volume
pseudo-Riemannian manifold. Hence, every isometric $G$-action on $M$
with a dense orbit is locally free (see \cite{Szaro,Zeghib}) and so
the orbits define a foliation that we will denote with $\OO$. The
bundle $T\OO$ tangent to the foliation $\OO$ is a trivial vector
bundle isomorphic to $M \times \g$, under the isomorphism $M \times \g
\rightarrow T\OO$ given by $(x, X) \mapsto X^*_x$.  This also defines
an isomorphism of every fiber $T_x\OO$ with $\g$.  We will refer to it
as the natural isomorphism between $T_x\OO$ and $\g$. Recall that, as
before and in the rest of this article, for $X$ in the Lie algebra of
a group acting on a manifold, we denote by $X^*$ the vector field on
the manifold whose one-parameter group of diffeomorphisms is given by
$(\exp(tX))_t$ through the action on the manifold.  On the other hand,
we will have the occasion to consider both left and right actions and
so our convention is to assume that an action is on the left unless
otherwise specified.

For any given pseudo-Riemannian manifold $N$, we will denote by
$\Kill(N)$ the globally defined Killing vector fields of $N$. Also, we
denote by $\Kill_0(N,x)$ the Lie algebra of globally defined Killing
vector fields that vanish at the given point $x$. The following result
is an easy application of the Jacobi identity and the fact that
Killing vector fields are derivations of the corresponding metric.
Also note that, in the rest of this work, for a vector space $W$ with
a non-degenerate symmetric bilinear form, we will denote with $\so(W)$
the Lie algebra of linear maps on $W$ that are skew-symmetric with
respect to the bilinear form.

\begin{lemma}\label{lemma-lambda}
  Let $N$ be a pseudo-Riemannian manifold and $x\in N$. Then, the map
  $ \lambda_x : \Kill_0(N,x) \rightarrow \so(T_x N)$ given by
  $\lambda_x(Z)(v) = [Z,V]_x$, where $V$ is any vector field such that
  $V_x = v$, is a well defined homomorphism of Lie algebras.
\end{lemma}

For any given point $x$ of a pseudo-Riemannian manifold, the map
$\lambda_x$ will denote from now on the homomorphism from the previous
lemma.

Gromov \cite{Gromov} proved that the presence of a geometric structure
of finite type (in the sense of Cartan) which is invariant under the
action of a simple Lie group yields large spaces of Killing vector
fields fixing given points in the manifold being acted upon. We refer
to \cite{Zimmer-rigid} for a detailed description of these
techniques. The statement below in the case of germs of Killing fields
is essentially contained in Section~9 of \cite{GCT} (see also
Proposition~2.3 in \cite{Quiroga-Z}). From this, the result for global
Killing vector fields is straightforward since $\widetilde{M}$ is
analytic and simply connected (see
\cite{GCT,Gromov,Zimmer-rigid}). Observe that, following our notation
with $M$, we denote with $\OO$ the foliation by $\widetilde{G}$-orbits
in $\widetilde{M}$.

\begin{proposition}\label{prop-g(x)}
  Let $G$ be a connected non-compact simple Lie group acting
  isometrically and with a dense orbit on a connected finite volume
  pseudo-Riemannian manifold $M$. Consider the $\widetilde{G}$-action
  on $\widetilde{M}$ lifted from the $G$-action on $M$. Assume that
  $M$ and the $G$-action on $M$ are both analytic. Then, there exists
  a conull subset $S\subset \widetilde{M}$ such that for every $x \in
  S$ the following properties are satisfied:
  \begin{enumerate}
  \item there is a homomorphism $\rho_x : \g \rightarrow
    \Kill(\widetilde{M})$ which is an isomorphism onto its image
    $\rho_x(\g) = \g(x)$.
  \item $\g(x)\subset \Kill_0(\widetilde{M},x)$, i.e.~every element of
    $\g(x)$ vanishes at $x$.
  \item For every $X,Y \in \g$ we have:
   $$
   [\rho_x(X),Y^*] = [X,Y]^* = -[X^*,Y^*].
   $$
   In particular, the elements in $\g(x)$ and their corresponding
   local flows preserve both $\OO$ and $T\OO^\perp$.
 \item The homomorphism of Lie algebras $\lambda_x\circ\rho_x : \g
   \rightarrow \so(T_x \widetilde{M})$ induces a $\g$-module structure
   on $T_x \widetilde{M}$ for which the subspaces $T_x \OO$ and $T_x
   \OO^\perp$ are $\g$-submodules.
  \end{enumerate}
\end{proposition}

With the above setup, assume that the $G$-orbits are non-degenerate
which, from now on, is considered with respect to the ambient
pseudo-Riemannian metric. In particular, the $\widetilde{G}$-orbits on
$\widetilde{M}$ are non-degenerate as well and we have a direct sum
decomposition $T\widetilde{M} = T\OO \oplus T\OO^\perp$. Hence, we can
consider the $\g$-valued $1$-form $\omega$ on $\widetilde{M}$ which is
given, at every $x \in \widetilde{M}$, by the composition
$T_x\widetilde{M} \rightarrow T_x\OO \cong \g$ of the natural
projection onto $T_x\OO$ and the natural isomorphism of this latter
space with $\g$. Also, consider the $\g$-valued $2$-form given by
$\Omega = d\omega|_{\wedge^2 T\OO^\perp}$. The following result is
elementary and a proof can be found in \cite{Quiroga-Z}.

\begin{lemma}\label{lemma-omega-Omega}
  Let $G$, $M$ and $S$ be as in Proposition~\ref{prop-g(x)}. If we
  assume that the $G$-orbits are non-degenerate, then:
  \begin{enumerate}
  \item For every $x \in S$, the maps $\omega_x : T_x\widetilde{M}
    \rightarrow \g$ and $\Omega_x : \wedge^2 T_x\OO^\perp \rightarrow
    \g$ are both homomorphisms of $\g$-modules, for the $\g$-module
    structures from Proposition~\ref{prop-g(x)}.

  \item The normal bundle $T\OO^\perp$ is integrable if and only if
    $\Omega = 0$.
  \end{enumerate}
\end{lemma}

The non-degeneracy of the orbits is ensured for low dimensional
manifolds by the next result, which appears in~\cite{Quiroga-Z} as
Lemma~2.7.

\begin{lemma}\label{lemma-nondeg}
  Let $G$ be a connected non-compact simple Lie group acting
  isometrically and with a dense orbit on a connected finite volume
  pseudo-Riemannian manifold $M$. If $\dim(M) < 2 \dim(G)$, then the
  bundles $T\OO$ and $T\OO^\perp$ have fibers that are non-degenerate
  with respect to the metric on $M$.
\end{lemma}

It turns out that for complete manifolds, if the $G$-orbits are
non-degenerate and the normal bundle to such orbits is integrable, then
the universal covering space can be split. Such a claim is the content
of the following proposition which is a particular case of Theorem~1.1
of~\cite{Quiroga-Z}. This result is in the spirit of similar ones
found in \cite{Cairns,Cairns-Ghys,Gromov}.

\begin{proposition}\label{prop-int-normal}
  Let $G$ be a connected non-compact simple Lie group acting
  isometrically on a connected complete finite volume
  pseudo-Riemannian manifold $M$. If the tangent bundle to the orbits
  $T\OO$ has non-degenerate fibers and the bundle $T\mathcal{O}^\perp$
  is integrable, then there is an isometric covering map
  $\widetilde{G}\times N \rightarrow M$ where the domain has the
  product metric for a bi-invariant metric on $\widetilde{G}$ and with
  $N$ a complete pseudo-Riemannian manifold.
\end{proposition}

As a consequence, we obtain a lower bound on the dimension of $M$.

\begin{proposition}\label{prop-bound}
  Let $M$ be a connected analytic pseudo-Riemannian manifold and $G$ a
  connected non-compact simple Lie group. Suppose that $M$ is complete
  weakly irreducible, has finite volume and admits an analytic
  isometric non-transitive $G$-action with a dense orbit. Then:
  \[
  \dim(M) \geq \dim(G) + m(\g),
  \]
  where $m(\g)$ is the dimension of the smallest non-trivial
  representation of $\g$ that admits an invariant non-degenerate
  symmetric bilinear form.
\end{proposition}
\begin{proof}
  Suppose that $\dim(M) < \dim(G) + m(\g)$. Since $m(\g) \leq \dim(G)$
  (the Killing form defines an inner product), by
  Lemma~\ref{lemma-nondeg} the bundle $T\OO^\perp$ has non-degenerate
  fibers with dimension $< m(\g)$. Hence, the definition of $m(\g)$
  implies that $T_x\OO^\perp$ is a trivial $\g$-module for the
  structure defined by Proposition~\ref{prop-g(x)}(4). Hence,
  Lemma~\ref{lemma-omega-Omega} yields the integrability of
  $T\OO^\perp$, and Proposition~\ref{prop-int-normal} contradicts the
  irreducibility of $M$.
\end{proof}

For a $G$-action as in Proposition~\ref{prop-g(x)}, consider
$\widetilde{M}$ endowed with the $\widetilde{G}$-action obtained by
lifting the $G$-action on $M$. With such setup, let us denote by $\HH$
the Lie subalgebra of $\Kill(\widetilde{M})$ consisting of the fields
that centralize the $\widetilde{G}$-action on $\widetilde{M}$. The
next result provides an embedding of $\g$ into $\HH$ that allows us to
apply representation theory to study the structure of $\HH$. We
observe that this statement is at the core of Gromov-Zimmer's
machinery on the study of actions preserving geometric structures (see
\cite{Gromov,Zimmer-rigid}).

\begin{lemma}\label{lemma-HH-module-struc}
  Let $S$ be as in Proposition~\ref{prop-g(x)}. Then, for every $x\in
  S$ and for $\rho_x$ given as in Proposition~\ref{prop-g(x)}, the
  map $\widehat{\rho}_x : \g \rightarrow \Kill(\widetilde{M})$ given
  by:
  \[
  \widehat{\rho}(X) = \rho_x(X) + X^*,
  \]
  is an injective homomorphism of Lie algebras whose image $\GG(x)$
  lies in $\HH$. In particular, $\widehat{\rho}_x$ induces on $\HH$ a
  $\g$-module structure such that $\GG(x)$ is a submodule isomorphic
  to $\g$.
\end{lemma}
\begin{proof}
  First, observe that the identity in Proposition~\ref{prop-g(x)}(3)
  is easily seen to imply that the image of $\widehat{\rho}_x$ lies in
  $\HH$.

  To prove that $\widehat{\rho}_x$ is a homomorphism of Lie algebras
  we apply Proposition~\ref{prop-g(x)}(3) as follows for $X,Y \in \g$:
  \begin{align*}
    [\widehat{\rho}_x(X),\widehat{\rho}_x(Y)]
    &= [\rho_x(X) + X^*,\rho_x(Y) + Y^*] \\
    &= [\rho_x(X),\rho_x(Y)] +
    [\rho_x(X),Y^*] + [X^*,\rho_x(Y)] + [X^*,Y^*] \\
    &= \rho_x([X,Y]) + [X,Y]^* + [X,Y]^* + [X^*,Y^*] \\
    &= \rho_x([X,Y]) + [X,Y]^* \\
    &= \widehat{\rho}_x([X,Y]).
  \end{align*}
  For the injectivity of $\widehat{\rho}_x$ we observe that
  $\widehat{\rho}_x(X) = 0$ implies $X^*_x = (\rho_x(X) + X^*)_x = 0$,
  which in turns yields $X = 0$ because the $G$-action is locally
  free. The last claim is now clear.
\end{proof}

We can now relate the $\g$-module structure of $\HH$ to that of
$T_x\widetilde{M}$.

\begin{lemma}\label{lemma-ev-homomorphism}
  Let $S$ be as in Proposition~\ref{prop-g(x)}. Consider
  $T_x\widetilde{M}$ and $\HH$ endowed with the $\g$-module structures
  given by Proposition~\ref{prop-g(x)}(4) and
  Lemma~\ref{lemma-HH-module-struc}, respectively. Then, for every
  $x\in S$, the evaluation map:
  \[
  ev_x : \HH \rightarrow T_x \widetilde{M}, \quad Z \mapsto Z_x,
  \]
is a homomorphism of $\g$-modules that satisfies
$ev_x(\GG(x))=T_x\OO$.  Furthermore, for almost every $x \in S$ we
have $ev_x(\HH) = T_x\widetilde{M}$.
\end{lemma}
\begin{proof}
  For every $x \in S$, if we let $Z\in \HH$ and $X \in \g$ be given,
  then:
  \begin{align*}
    ev_x(X\cdot Z) &= [\widehat{\rho}_x(X),Z]_x
    = [\rho_x(X) + X^*,Z]_x \\
    &= [\rho_x(X),Z]_x = X\cdot Z_x = X\cdot ev_x(Z)
  \end{align*}
  where we have used Lemma~\ref{lemma-lambda} and the definition of
  the $\g$-module structures involved, thus proving the first part.
  The last claim follows by an easy adaptation of the proof of
  Lemma~4.1 of \cite{Zimmer-entropy}, which establishes the
  transitivity of $\HH$ on an open conull dense subset of
  $\widetilde{M}$.
\end{proof}

To study in the following sections those $G$-actions for which
$T\OO^\perp$ is non-integrable we will need to use some known results
that relate isometries with Killing fields for complete manifolds.
First, we have the following result, which follows from the rigidity
(in the sense of \cite{Gromov}) of pseudo-Riemannian metrics and their
basic properties.

\begin{lemma}\label{lemma-analytic-symmetries}
  Let $N$ be an analytic pseudo-Riemannian manifold. Then, every
  Killing vector field of $N$, either local or global, is analytic.
  In particular, the isometry group $\Iso(N)$ acts analytically on N.
\end{lemma}

By Proposition~30 of Chapter~9 from \cite{ONeill-book}, on a complete
pseudo-Riemannian manifold every global Killing vector field is
complete. Hence, Proposition~33 of Chapter~9 from \cite{ONeill-book}
has the following immediate consequence.

\begin{lemma}\label{lemma-Kill-Iso}
  Let $N$ be a complete pseudo-Riemannian manifold and suppose that
  the action of its isometry group $\Iso(N)$ is considered on the
  left. If $\iso(N)$ denotes the Lie algebra of $\Iso(N)$, then the
  map:
  \[
  \iso(N) \rightarrow \Kill(N), \quad X \mapsto X^*,
  \]
  is an anti-isomorphism of Lie algebras. In particular, $[X,Y]^* =
  -[X^*,Y^*]$ for every $X,Y \in \iso(N)$.
\end{lemma}

We now use the above to prove that on a complete manifold every Lie
algebra of Killing fields can be realized from an isometric right
action.

\begin{lemma}\label{lemma-Kill-to-action}
  Let $N$ be a complete pseudo-Riemannian manifold and $H$ a simply
  connected Lie group with Lie algebra $\h$. If $\psi : \h \rightarrow
  \Kill(N)$ is a homomorphism of Lie algebras, then there exists an
  isometric right $H$-action $N\times H \rightarrow N$ such that
  $\psi(X) = X^*$, for every $X \in \h$. Furthermore, if $N$ is
  analytic, then the $H$-action is analytic as well.
\end{lemma}
\begin{proof}
  Consider the map $\alpha : \iso(N) \rightarrow \Kill(N)$ given by
  $\alpha(Y) = -Y^*$, which is an isomorphism of Lie algebras by
  Lemma~\ref{lemma-Kill-Iso}. Let $\Psi : H \rightarrow \Iso(N)$ be
  the homomorphism of Lie groups induced by the homomorphism
  $\alpha^{-1} \circ \psi : \h \rightarrow \iso(n)$. This yields a
  smooth isometric right $H$-action given by:
  \[
  N \times H \rightarrow N, \quad (n,h) \mapsto nh = \Psi(h^{-1})(n).
  \]
  For $X\in \h$ there is $Y\in \iso(N)$ such that $\psi(X) =
  -\alpha(Y) = Y^*$. Hence, for the right $H$-action one computes
  $X^*$ at every $p\in N$ as follows:
  \begin{alignat*}{2}
    X^*_p &= \frac{d}{dt}\Big|_{t=0} p \exp(tX)
    = \frac{d}{dt}\Big|_{t=0}\Psi(\exp(-tX))(p) \\
    &= \frac{d}{dt}\Big|_{t=0}\exp(-t(\alpha^{-1}\circ\psi)(X))p
    = \frac{d}{dt}\Big|_{t=0}\exp(tY)p \\
    &= Y^*_p = \psi(X)_p,
  \end{alignat*}
  which proves the first part of the lemma. Note that the first and
  second to last identities use the definition of $Z^*$ for the right
  $H$-action and the left $\Iso(N)$-action, respectively.  Finally,
  the last part of our statement follows from the last claim of
  Lemma~\ref{lemma-analytic-symmetries}.
\end{proof}

\section{Structure of the centralizer of isometric
  $\uSO_0(p,q)$-actions for low dimensional $M$ and non-integrable
  $T\OO^\perp$}
\label{section-structure-centralizer}
\noindent
For $n\in\Z_+$ let $p,q\in\Z_+$ be such that $p+q=n$ and
\[
I_{p,q} = \left(
  \begin{matrix}
    I_p & 0 \\
    0 & -I_q
  \end{matrix}
\right).
\]
Then, $\so(p,q)$ is the Lie algebra of linear transformations of
$\R^{n}$ that are anti-symmetric with respect to the inner product
$\left<\cdot,\cdot\right>_{p,q}$ on $\R^n$ defined by $I_{p,q}$,
$\SO_o (p,q)$ denotes the connected group of isometries with respect
to $\left<\cdot ,\cdot \right>_{p,q}$ and $\uSO_0(p,q)$ its universal
covering group.
We let $\R^{p,q}$ denote the $\so(p,q)$-module defined by the natural
representation of $\so(p,q)$ in $\R^{n}$.  Denote by $C^+$ and $C^-$
the $\so(4,4)$-modules given by real forms of the two half spin
representations of $\so(8,\C)$. We refer to $C^+$ and $C^-$ as the
half spin representations of $\so(4,4)$.

In preparation for the proof of the Main Theorem, we assume in this
section that $M$ is a connected finite volume analytic
pseudo-Riemannian manifold with $\dim(M) = n(n + 1)/2$. We assume that
$p+q \geq 5$ or $(p,q) = (3,1)$.
We also assume that $\uSO_0(p,q)$ acts analytically, isometrically and
with a dense orbit on $M$. In particular, by Lemma~\ref{lemma-nondeg}
we have the direct sum $TM = T\OO\oplus T\OO^\perp$. Also note that
$T\OO^\perp$ has rank $n$. Finally, we also assume in the rest of this
section that the bundle $T\OO^\perp$ is non-integrable.

As before, the $\uSO_0(p,q)$-action on $M$ can be lifted to
$\widetilde{M}$, thus inducing a direct sum decomposition
$T\widetilde{M} = T\OO \oplus T\OO^\perp$ as a consequence of the
corresponding property for $M$. Again, we denote with $\OO$ the
foliation in $\widetilde{M}$ whose leaves are the orbits for the
$\uSO_0(p,q)$-action on $\widetilde{M}$.

In the rest of this work we will denote with $\HH$ the Lie subalgebra
of $\Kill(\widetilde{M})$ consisting of the fields that centralize the
$\uSO_0(p,q)$-action. In particular, there is a set $S$ as in
Proposition~\ref{prop-g(x)} for which
Lemmas~\ref{lemma-HH-module-struc} and \ref{lemma-ev-homomorphism}
hold. We will now see that our hypotheses allow to provide a precise
description of the $\so(p,q)$-module structures obtained through these
results from the previous section.

\begin{lemma}\label{lemma-TxOperp-is-Rpq}
  Let $S$ be as in Proposition~\ref{prop-g(x)}. Consider
  $T_x\OO^\perp$ endowed with the $\so(p,q)$-module structure given by
  Proposition~\ref{prop-g(x)}(4). Then, for almost every $x \in S$:
  \begin{enumerate}
  \item for $(p,q) \not= (4,4)$, the $\so(p,q)$-module $T_x\OO^\perp$
    is isomorphic to $\R^{p,q}$, and
  \item for $p = q = 4$, the $\so(4,4)$-module $T_x\OO^\perp$ is
    isomorphic to either $\R^{4,4}$, $C^+$ or $C^-$.
  \end{enumerate}
  In particular, $\so(T_x\OO^\perp)$ is isomorphic to $\so(p,q)$ as a
  Lie algebra and as a $\so(p,q)$-module, for almost every $x \in S$.
\end{lemma}
\begin{proof}
  By Lemma~\ref{lemma-omega-Omega}(2) and since we are assuming that
  $T\OO^\perp$ is non-integrable, the $2$-form $\Omega$ considered in
  its statement is non-zero. This $2$-form is clearly analytic and
  thus it vanishes at a proper analytic subset of $\widetilde{M}$
  which is necessarily null. Hence, $\Omega_x \neq 0$ for almost every
  $x \in S$. Let us choose and fix $x\in S$ such that $\Omega_x \neq
  0$; we will prove that the conclusions of the statement hold for
  such $x$.

  Lemma~\ref{lemma-omega-Omega}(1) implies that the map $\Omega_x :
  \wedge^2 T_x\OO^\perp \rightarrow \so(p,q)$ is a homomorphism of
  $\so(p,q)$-modules, which is then non-trivial by our choice of $x$.
  Since $\dim(T_x\OO^\perp) = n$ and because $\so(p,q)$ is an
  irreducible module, it follows that $\Omega_x$ is an
  isomorphism. Then, the irreducibility of $\so(p,q)$ implies that
  $T_x\OO^\perp$ is irreducible as well.

  By Lemma~\ref{lemma-Rpq-smallest} it follows that
  $T_x\OO^\perp\simeq \R^{p,q}$ except for the cases given by the Lie
  algebras $\so (3,1)$ and $\so(4,4)$. For these Lie algebras the
  other possibilities are $\C^2_\R$ for $\so(3,1)$, and real forms
  $C^+$ and $C^-$ of the two half spin representations of $\so(8,\C)$,
  for the case of $\so(4,4)$.

  Let us consider the case of $\so (3,1)$. If $(\pi ,V)$ is a
  $G$-module and $\chi_V =\Tr\circ \pi$ is its character, then for
  every $g \in G$ we have $\chi_{\wedge^2 V}(g) =
  \frac{1}{2}\left(\chi_V (g)^2-\chi_V (g^2)\right)$. For our setup,
  we need to determine for which of $V \simeq \R^{3,1}$ or $V \simeq
  \C^2_\R$ we have $\chi_{\so (3,1)}(g)=\frac{1}{2}\left(\chi_V
    (g)^2-\chi_V (g^2)\right)$. A simple calculation using
  \[
  g=
  \begin{pmatrix}
    \cosh (t) & 0 &\sinh (t)\\
    0 & I_2 & 0\\
    \sinh (t) & 0 &\cosh (t)
  \end{pmatrix}
  \]
  shows that the above holds only for $V \simeq \R^{3,1}$, and so
  $T_x\OO^\perp \simeq \R^{3,1}$.

  For the final claim, we observe that the representation of
  $\so(p,q)$ on $T_x \OO^\perp$ defines a non-trivial homomorphism
  $\so(p,q) \rightarrow \so(T_x\OO^\perp)$. Since $\so(p,q)$ is
  simple, the latter is injective and so it is an isomorphism.
\end{proof}

The previous results allow us to obtain the following decomposition of
the centralizer $\HH$ of the $\uSO_0(p,q)$-action into submodules
related to the geometric structure on $\widetilde{M}$.

\begin{lemma}\label{lemma-HH-decomposition}
  Let $S$ be as in Proposition~\ref{prop-g(x)}. Then, for almost every
  $x \in S$ and for the $\so(p,q)$-module structure on $\HH$ from
  Lemma~\ref{lemma-HH-module-struc} there is a decomposition into
  $\so(p,q)$-submodules $\HH = \GG(x)\oplus\HH_0(x)\oplus\VV(x)$,
  satisfying:
  \begin{enumerate}
  \item $\GG(x) = \widehat{\rho}_x(\so(p,q))$ is a Lie subalgebra of
    $\HH$ isomorphic to $\so(p,q)$ and $ev_x(\GG(x)) = T_x\OO$.
  \item $\HH_0(x) = \mathrm{ker}(ev_x)$, which is either $0$ or a Lie
    subalgebra of $\HH$ isomorphic to $\so(p,q)$. In the latter case,
    $\HH_0(x)$ is also isomorphic to $\so(p,q)$ as a
    $\so(p,q)$-module.
  \item $ev_x(\VV(x)) = T_x\OO^\perp$ and
    \begin{enumerate}
    \item for $(p,q) \not= (4,4)$, $\VV(x)$ is isomorphic to
      $\R^{p,q}$ as $\so(p,q)$-module,
    \item for $p = q = 4$, $\VV(x)$ is isomorphic to either
      $\R^{4,4}$, $C^+$ or $C^-$ as $\so(4,4)$-module.
    \end{enumerate}
  \end{enumerate}
  In particular, the evaluation map $ev_x$ defines an isomorphism of
  $\so(p,q)$-modules $\GG(x) \oplus \mathcal{V}(x) \rightarrow T_x
  \widetilde{M} = T_x \OO \oplus T_x \OO^\perp$ preserving the
  summands in that order.
\end{lemma}
\begin{proof}
  Note that the conclusions of both Lemmas~\ref{lemma-ev-homomorphism}
  and \ref{lemma-TxOperp-is-Rpq} are satisfied for almost every $x \in
  S$. Let us choose and fix one such point $x \in S$.  By
  Lemma~\ref{lemma-HH-module-struc}, we conclude that $\GG(x) =
  \widehat{\rho}_x(\so(p,q))$ is indeed a Lie subalgebra isomorphic to
  $\so(p,q)$.

  Define $\HH_0(x) = \mathrm{ker}(ev_x)$. By
  Lemma~\ref{lemma-ev-homomorphism}, it follows that $\HH_0(x)$ is an
  $\so(p,q)$-submodule of $\HH$. Moreover, since $\mathcal{H}_0(x) =
  \HH\cap\Kill_0(\widetilde{M},x)$ it follows that it is a Lie
  subalgebra as well.

  On the other hand, the elements of $\GG(x)$ are of the form
  $\rho_x(X) + X^*$, with $X \in \so(p,q)$, where $\rho_x$ is the Lie
  algebra homomorphism from Proposition~\ref{prop-g(x)}. Hence, for
  any such element we have $ev_x(\rho_x(X) + X^*) = X^*_x$; in
  particular, the condition $ev_x(\rho_x(X) + X^*) = 0$ implies
  $X=0$. In other words, $\GG(x)\cap\HH_0(x) = \{0\}$. Hence, there
  exists an $\so(p,q)$-submodule $\VV'(x)$ complementary to
  $\GG(x)\oplus\HH_0(x)$ in $\HH$. In particular, $ev_x$ restricted to
  $\GG(x)\oplus\VV'(x)$ is an isomorphism of $\so(p,q)$-modules onto
  $T_x\widetilde{M}$. Hence, if we choose $\VV(x)$ as the inverse
  image of $T_x\OO^\perp$ under this isomorphism, then we have the
  required decomposition into $\so(p,q)$-submodules except for the
  properties of $\HH_0(x)$, which we now proceed to consider.

  Let $\Kill_0(\widetilde{M},x,\OO)$ be the Lie algebra of Killing
  vector fields on $\widetilde{M}$ that preserve the foliation $\OO$
  and that vanish at $x$. Hence, every Killing field in
  $\Kill_0(\widetilde{M},x,\OO)$ leaves invariant the normal bundle
  $T\OO^\perp$, and so the map $\lambda_x$ from
  Lemma~\ref{lemma-lambda} induces a homomorphism of Lie algebras:
  \[
  \lambda_x^\perp : \Kill_0(\widetilde{M},x,\OO)
  \rightarrow \so(T_x\OO^\perp), \quad
  X \mapsto \lambda_x(X)|_{T_x\OO^\perp}.
  \]
  We observe that both Lie algebras $\rho_x(\so(p,q)) = \so(p,q)(x)$
  and $\HH_0(x)$ lie inside of $\Kill_0(\widetilde{M},x,\OO)$. A
  number of properties for the restriction of $\lambda_x^\perp$ to
  $\so(p,q)(x)$ and $\HH_0(x)$ will imply the needed conditions on
  $\HH_0(x)$.

  {\em Claim 1: $\lambda_x^\perp$ restricted to $\HH_0(x)$ is
    injective.} We recall that a Killing vector field is completely
  determined by its $1$-jet at $x$; this is a consequence of the fact
  that pseudo-Riemannian metrics are $1$-rigid (see
  \cite{GCT,Gromov}). If $Z \in \HH_0(x)$ is given, then $Z_x = 0$ and
  so it is completely determined by the values of $[Z,V]_x$ for $V$ a
  vector field in a neighborhood of $x$. On the other hand
  $[Z,X^*]_x=0$ for $X \in \so(p,q)$ and so $[Z,V]_x = 0$ whenever
  $V_x\in T_x\OO$. We conclude that $Z$ is completely determined by
  the values of $[Z,V]_x$ for $V$ such that $V_x \in T_x \OO^\perp$.
  In other words, $[Z,V]_x=0$ for every $V_x\in T_x\OO^\perp$ implies
  $Z=0$. This yields the injectivity of $\lambda_x^\perp$ on
  $\HH_0(x)$.

  {\em Claim 2: $\lambda_x^\perp(\so(p,q)(x)) = \so(T_x\OO^\perp)$.}
  By Proposition~\ref{prop-g(x)}(4), the vector space $T_x\OO^\perp$
  has a $\so(p,q)$-module structure induced from the homomorphism
  $\lambda_x^\perp\circ\rho_x$. By our choice of $x$ and
  Lemma~\ref{lemma-TxOperp-is-Rpq} such module structure is in fact
  non-trivial. Hence, $\lambda_x^\perp\circ\rho_x : \so(p,q)
  \rightarrow \so(T_x\OO^\perp)$ is non-trivial as well and so it is
  injective. But then it has to be surjective because the domain and
  target have the same dimensions.

  {\em Claim 3: $\lambda_x^\perp(\HH_0(x))$ is an ideal in
    $\so(T_x\OO^\perp)$.} Let $Z\in \HH_0(x)$ and $T \in
  \so(T_x\OO^\perp)$ be given. Then, by Claim 2, there is some $X \in
  \so(p,q)$ such that $T = \lambda_x^\perp(\rho_x(X))$. For every
  local vector field $V$ such that $V_x \in T_x \OO^\perp$ we have:
  \begin{multline*}
    [T,\lambda_x^\perp(Z)](V_x) =
    [\lambda_x^\perp(\rho_x(X)),\lambda_x^\perp(Z)](V_x) =
    [\rho_x(X),[Z,V]]_x - [Z,[\rho_x(X),V]]_x \\
    = [[\rho_x(X),Z],V]_x = [[\rho_x(X) + X^*,Z],V]_x =
    [[\widehat{\rho}_x(X),Z],V]_x.
  \end{multline*}
  Since the $\so(p,q)$-module structure on $\HH$ is defined by
  $\widehat{\rho}_x$ and $\HH_0(x)$ is a submodule of such structure,
  we have $[\widehat{\rho}_x(X),Z] \in \HH_0(x)$, and so the last
  formula proves that $[T,\lambda_x^\perp(Z)] =
  \lambda_x^\perp([\widehat{\rho}_x(X),Z])$, thus showing the claim.

  Claim 1 shows that $\HH_0(x)$ is a Lie algebra isomorphic to its
  image in $\so(T_x\OO^\perp)$ under $\lambda_x^\perp$. Such image is
  by Claim 3 an ideal of $\so(T_x\OO^\perp)$. By our choice of $x$ and
  Lemma~\ref{lemma-TxOperp-is-Rpq}, the Lie algebra
  $\so(T_x\OO^\perp)$ is isomorphic to $\so(p,q)$, which is simple
  since $n \geq 4$ and $(p,q) \not= (2,2)$. This implies that
  $\HH_0(x)$ is either $0$ or isomorphic to $\so(p,q)$ as a Lie
  subalgebra of $\HH$.

  On the other hand, for $X \in \so(p,q)$ and $Z \in \HH_0(x)$,
  considering the definitions of the $\so(p,q)$-module structures
  involved we have:
  \begin{align*}
   \lambda_x^\perp(X\cdot Z) &= \lambda_x^\perp([\widehat{\rho}_x(X),Z])
  = \lambda_x^\perp([\rho_x(X),Z]) \\
  &= [\lambda_x^\perp(\rho_x(X)),\lambda_x^\perp(Z)]
  = X\cdot \lambda_x^\perp(Z),
  \end{align*}
  where the second identity holds by the definition of
  $\widehat{\rho}_x$ in terms of $\rho_x$ and because $\HH_0(x)$
  centralizes the $\uSO_0(p,q)$-action. But this last relation shows
  that $\lambda_x^\perp$ restricted to $\HH_0(x)$ is a homomorphism of
  $\so(p,q)$-modules. By Lemma~\ref{lemma-TxOperp-is-Rpq} we conclude
  that $\HH_0(x)$ is either $0$ or isomorphic to $\so(p,q)$ as a
  $\so(p,q)$-module.
\end{proof}

We now obtain a description of the Lie algebra structure
of the centralizer $\HH$.

\begin{lemma}\label{lemma-HH-Lie-structure}
  Let $S$ be as in Proposition~\ref{prop-g(x)}. With the notation from
  Lemma~\ref{lemma-HH-decomposition}, one of the following conditions
  is satisfied for almost every $x \in S$:
  \begin{enumerate}
  \item The radical $\rad(\HH)$ of $\HH$ is Abelian and $\rad(\HH) =
    \VV(x)$.
  \item $\HH_0(x) \not= \{0\}$ and $\HH$ is the sum of two simple
    ideals one of them being $\HH_0(x)\oplus\VV(x)$.
  \item $\HH_0(x) = \{0\}$ and $\HH = \GG(x)\oplus\VV(x)$ is
    isomorphic as a Lie algebra to either $\so(p,q+1)$ or
    $\so(p+1,q)$.
  \end{enumerate}
\end{lemma}
\begin{proof}
  For this proof let us choose and fix $x \in S$ satisfying the
  conclusions of Lemmas~\ref{lemma-HH-module-struc},
  \ref{lemma-ev-homomorphism}, \ref{lemma-TxOperp-is-Rpq} and
  \ref{lemma-HH-decomposition}.

  Note that $\GG(x)\oplus\HH_0(x)$ is a Lie subalgebra with $\HH_0(x)$
  as an ideal because $[\GG(x),\HH_0(x)] \subset \HH_0(x)$ as a
  consequence of Lemmas~\ref{lemma-HH-module-struc} and
  \ref{lemma-HH-decomposition}. Since $\GG(x),\HH_0(x)$ are both
  isomorphic to $\so(p,q)$ if $\HH_0(x) \not= \{0\}$, we conclude that
  $\GG(x)\oplus\HH_0(x)$ is semisimple.  Furthermore, this implies
  that $\GG(x)\oplus\HH_0(x)$ is isomorphic to either $\so(p,q)$ or
  $\so(p,q)\oplus\so(p,q)$ according to whether $\HH_0(x)$ is zero or
  not, respectively.

  Choose a Levi factor $\s$ of $\HH$ that contains the Lie subalgebra
  $\GG(x)\oplus\HH_0(x)$. Since the $\so(p,q)$-module structure of
  $\HH$ is defined by the Lie subalgebra $\GG(x)$ (see
  Lemma~\ref{lemma-HH-module-struc}), it follows that $\s$ is an
  $\so(p,q)$-submodule of $\HH$. Let $W$ be an $\so(p,q)$-submodule of
  $\HH$ such that $\s = \GG(x)\oplus\HH_0(x)\oplus W$.  Since
  $\rad(\HH)$ is an ideal, this induces the following decomposition of
  $\HH$ as a direct sum of $\so(p,q)$-submodules:
  $$
  \HH = \GG(x)\oplus\HH_0(x)\oplus W \oplus \rad(\HH).
  $$
  {}From this decomposition of $\so(p,q)$-modules, as well as
  Lemmas~\ref{lemma-HH-module-struc} and \ref{lemma-HH-decomposition},
  we conclude that one of the following holds:
  \begin{enumerate}
  \item[(a)] $\s = \GG(x)\oplus\HH_0(x)$ and $\rad(\HH) = \VV(x)$, or
  \item[(b)] $\rad(\HH) = \{0\}$ and so $\HH$ is semisimple.
  \end{enumerate}

  Let us assume that case (a) holds. Then, the Lie brackets of $\HH$
  restricted to $\wedge^2\VV(x)$ define, by the Jacobi identity and
  Lemma~\ref{lemma-HH-module-struc}, a homomorphism of
  $\so(p,q)$-modules $\wedge^2\VV(x) \rightarrow \VV(x)$. This
  homomorphism is necessarily trivial since $\VV(x)$ is
  $n$-dimensional, $\wedge^2\VV(x) \simeq\so(p,q)$ is irreducible and
  $n \geq 4$. This shows that $\VV(x) = \rad(\HH)$ is Abelian and
  yields (1) from our statement.

  Let us now assume that case (b) holds, and write $\HH = \h_1 \times
  \cdots \times \h_k$ a direct product of simple ideals. Since each
  such ideal is invariant by $\GG(x)$, it is an $\so(p,q)$-submodule,
  and so it follows that $k \leq 3$ since the decomposition of $\HH$
  from Lemma~\ref{lemma-HH-decomposition} has at most $3$ irreducible
  summands. Moreover, if $k = 3$ we conclude that, after reindexing
  the ideals, we have $\VV(x) = \h_1$ and $\GG(x)\oplus\HH(x) =
  \h_2\times\h_3$. In particular, $[\GG(x),\VV(x)] = \{0\}$ which
  implies that $\VV(x)$ is a trivial $\so(p,q)$-submodule and
  contradicts Lemma~\ref{lemma-HH-decomposition}(3). Hence we can
  further assume that the number of simple ideals of $\HH$ is $k \leq
  2$.

  First suppose that $\HH = \h_1\times\h_2$, the direct product of two
  simple ideals. If $\HH_0(x) = \{0\}$, then the decomposition of
  $\HH$ from Lemma~\ref{lemma-HH-decomposition} has two irreducible
  summands and we can reindex the ideals $\h_1,\h_2$ to assume that
  $\h_1 = \GG(x)$ and $\h_2 = \VV(x)$. But this implies that
  $[\GG(x),\VV(x)] = \{0\}$, a contradiction. Hence we conclude that
  $\HH_0(x) \not= \{0\}$ in the current case. In particular, $\HH$ is
  the direct sum of three irreducible $\so(p,q)$-submodules. Hence,
  after decomposing $\h_1,\h_2$ as the direct sum of irreducible
  $\so(p,q)$-submodules, and reindexing if necessary, we can assume
  that $\h_1$ is an irreducible $\so(p,q)$-submodule and that $\h_2 =
  V_1 \oplus V_2$, where $V_1,V_2$ are irreducible
  $\so(p,q)$-submodules. By comparing the decomposition $\HH = \h_1
  \oplus V_1 \oplus V_2$, with the one from
  Lemma~\ref{lemma-HH-decomposition}, we conclude that $\VV(x)$ is
  either one of $\h_1$, $V_1$ or $V_2$. If $\VV(x) = \h_1$, then
  $[\VV(x),\VV(x)] \subset \VV(x)$ and an argument used above shows
  that $\VV(x)$ is Abelian, which contradicts the simplicity of
  $\h_1$. Hence, without loss of generality, we can assume that
  $\VV(x) = V_2$, and so that $\GG(x) \oplus \HH_0(x) = \h_1 \oplus
  V_1$. In particular, since $V_1$ is a subspace of both of the Lie
  algebras $\GG(x) \oplus \HH_0(x) = \h_1\oplus V_1$ and $\h_2 = V_1
  \oplus V_2$ it follows that $[V_1,V_1] \subset V_1$, thus showing
  that $V_1$ itself is a Lie algebra. But since $[\h_1,V_1] = \{0\}$,
  this implies that the right-hand side of the sum:
  \[
  \GG(x) \oplus  \HH_0(x) = \h_1 \oplus V_1
  \]
  is the decomposition into simple ideals. Since $\HH_{0}(x)$ is an
  ideal of $\GG(x)\oplus\HH_0(x)$, it is either $\h_1$ or $V_1$. If
  $\HH_0(x) = \h_1$, then $[\HH_0(x),\VV(x)] = \{0\}$, which is in
  contradiction with Claim~1 in the proof of
  Lemma~\ref{lemma-HH-decomposition} because $ev_x(\VV(x)) = T_x
  \OO^\perp$. Hence, $\HH_0(x) = V_1$ and so $\HH_0(x)\oplus\VV(x) =
  \h_2$ is a simple ideal of $\HH$, thus establishing option (2).

  Finally, let us assume that $\HH$ is a simple Lie algebra. We will
  prove that in this case (3) holds.

  Let us start by assuming that $\HH_0(x) \not= \{0\}$. Hence, from
  the above remarks, we can write $\GG(x)\oplus\HH_0(x) =
  \g_1\oplus\g_2$, where $\g_1, \g_2$ are ideals of
  $\GG(x)\oplus\HH_0(x)$ both isomorphic to $\so(p,q)$. Let $V$ be a
  $\GG(x)\oplus\HH_0(x)$-submodule of $\HH$ such that:
  \[
  \HH = \GG(x)\oplus\HH_0(x)\oplus V.
  \]
  In particular, $V$ has dimension $n$. Moreover, $V$ is necessarily a
  non-trivial $\g_1$-module, since otherwise $\g_1$ would be a proper
  ideal of $\HH$. Then, Lemma~\ref{lemma-Rpq-smallest} implies that we
  can decompose $V = V_0 \oplus V_1$ where $V_0$ is a trivial
  $\g_1$-module and $V_1$ is an irreducible $\g_1$-module. Note that
  this can be done so that $V_0 = \{0\}$, except for the cases given
  by $\so(3,2)$ and $\so(3,3)$ for which we can assume $\dim(V_0) = 1$
  or $2$, respectively. In any case, this yields a decomposition of
  $\HH$ into $\g_{1}$-submodules given by:
  \[
  \HH = \g_{1}\oplus\g_{2}\oplus V_{0} \oplus V_{1}.
  \]
  Since $\g_1$ and $\g_2$ commute with each other, then for every $X
  \in \g_2$ the map $\ad_\HH(X)$ defines a $\g_1$-module homomorphism
  of $V$ and so preserves its summands corresponding to given
  isomorphism classes for the $\g_{1}$-module structure. Hence,
  $V_{0}$ and $V_{1}$ are $\g_{2}$-modules as well, and by
  Lemma~\ref{lemma-Rpq-smallest} it follows that $V_{0}$ is a trivial
  $\g_{2}$-module, because $\dim(V_{0}) \leq 2$. As before, $V_{1}$ is
  a non-trivial $\g_{2}$-module, since otherwise $\g_{2}$ would be a
  proper ideal of $\HH$.

  The above shows that $\ad_\HH$ restricted to $\GG(x)\oplus\HH_0(x)
  \simeq \so(p,q)\oplus\so(p,q)$ leaves invariant $V_1$ and induces a
  representation:
  \[
  \rho: \so(p,q)\oplus\so(p,q) \rightarrow \gsl(V_1) \simeq \gsl(k,\R),
  \]
  for some $k \leq n$, which is injective when restricted to each
  summand. Furthermore, $\rho(\so(p,q)\oplus \{0\})$ and
  $\rho(\{0\}\oplus \so(p,q))$ centralize each other, from which the
  simplicity of $\so(p,q)$ implies that $\rho(\so(p,q)\oplus \{0\})
  \cap \rho(\{0\}\oplus\so(p,q)) = \{0\}$. In particular, $\rho$
  realizes $\so(p,q)\oplus\so(p,q)$ as a Lie subalgebra of
  $\gsl(k,\R)$. Note that the codimension of $\so(p,q)\oplus\so(p,q)$
  in $\gsl(k,\R)$ is $k^2 - 1 - n(n-1) \leq n-1$.

  Replacing $\HH$ with $\gsl(k,\R)$ and repeating these arguments,
  more than once if necessary, yields a non-trivial representation of
  $\so(p,q)$ with dimension strictly smaller than the lower bound
  obtained in Lemma~\ref{lemma-Rpq-smallest} for irreducible
  non-trivial representations. This contradiction proves that
  $\HH_0(x) = \{0\}$.

  With our assumption that $\HH$ is simple we thus obtain $\HH =
  \GG(x) \oplus \VV(x)$. Hence, to conclude (3) it remains to show
  that $\HH$ is isomorphic to either $\so(p,q+1)$ or $\so(p+1,q)$.

  Recall from the previous remarks that $\HH = \GG(x) \oplus \VV(x)$
  is the decomposition into irreducible $\so(p,q)$-modules for an
  isomorphism $\GG(x) \simeq \so(p,q)$ of Lie algebras. Furthermore,
  either $\VV(x) \simeq \R^{p,q}$ as $\so(p,q)$-modules or $p = q = 4$
  and $\VV(x)$ is isomorphic to one of $C^+$ or $C^-$ as
  $\so(4,4)$-modules.  Let us first assume that the latter case
  holds. Then, $\HH$ is a $36$-dimensional simple Lie algebra. Since
  there are no simple complex Lie algebras of dimension $18$, it
  follows that the complexification $\HH^\C$ is a simple Lie
  algebra. Hence, a direct inspection of the simple complex Lie
  algebras shows that $\HH^\C$ is up to isomorphism either $\so(9,\C)$
  or $\symp(4,\C)$. For $\HH^\C \simeq \so(9,\C)$ we conclude that
  $\HH \simeq \so(5,4)$ since it contains the Lie algebra $\GG(x)
  \simeq \so(4,4)$. For the case $\HH^\C \simeq \symp(4,\C)$ we obtain
  a non-trivial homomorphism $\so(8,\C) \simeq \GG(x)^\C \subset
  \HH^\C \simeq \symp(4,\C)$, which yields an $8$-dimensional
  $\so(8,\C)$-module, non-trivial and so irreducible, with an
  invariant non-degenerate skew-symmetric form. The latter is a
  contradiction since every $8$-dimensional irreducible
  $\so(8,\C)$-module carries a unique (up to a constant) invariant
  non-degenerate symmetric form (see page~217 of
  \cite{BourbakiLie7-9}).

  Hence, we can assume that $\VV(x) \simeq \R^{p,q}$ as
  $\so(p,q)$-modules. Also, since the $\so(p,q)$-module structure on
  $\HH$ is induced by $\GG(x)$, we have $[\GG(x),\VV(x)] \subset
  \VV(x)$. On the other hand, the Lie brackets and the projection $\HH
  \rightarrow \VV(x)$ define a homomorphism of $\so(p,q)$-modules
  $\wedge^2\VV(x) \rightarrow \VV(x)$, which is thus trivial. This
  implies that $[\VV(x),\VV(x)] \subset \GG(x)$.  Hence, there exists
  a linear isomorphism:
  \[
  \varphi : \HH = \GG(x)\oplus\VV(x) \rightarrow
  \so(p,q)\oplus\R^{p,q},
  \]
  that preserves the summands in that order, that restricts to an
  isomorphism of Lie algebras $\GG(x) \rightarrow \so(p,q)$ and that
  defines an isomorphism of $\so(p,q)$-modules.  Moreover, we also
  have proved the relations:
  \[
    [\GG(x),\GG(x)] \subset \GG(x),\quad
    [\GG(x),\VV(x)] \subset \VV(x),\quad
    [\VV(x),\VV(x)] \subset \GG(x).
  \]
  Note that the last relation defines an isomorphism $T :
  \wedge^2\VV(x) \rightarrow \GG(x)$ of $\GG(x)$-modules; otherwise,
  $\VV(x)$ would be a non-trivial Abelian ideal of $\HH$. Recall the
  map
  $T_c(u\wedge v) = c(\left<u,\cdot\right>_{p,q}v -
  \left<v,\cdot\right>_{p,q}u)$ from Lemma~\ref{lemma-Rpq-so(p,q)}. By
  Lemma~\ref{lemma-Rpq-so(p,q)} and with respect to the isomorphism
  $\varphi$, the map $T$ is of the form $T_c$ for some $c\in
  \R\setminus\{0\}$, if $n \geq 5$. Let us now assume that $n \geq
  5$. Hence, the Lie algebra structure on $\so(p,q)\oplus\R^{p,q}$
  induced by $\varphi$ is given by $[\cdot,\cdot]_c$, as in
  Lemma~\ref{lemma-Lie-so(p,q)Rpq}, for some $c \neq 0$, and so
  isomorphic to either $\so(p,q+1)$ or $\so(p+1,q)$. This shows that
  (3) holds when $\HH$ is simple and $(p,q) \not= (3,1)$.

  Finally, consider the case $(p,q) = (3,1)$. Under such assumption,
  we have $\dim(\HH) = \dim(\GG(x)) + \dim(\VV(x)) = 10$. Since there
  is no simple complex Lie algebra of dimension $5$ we conclude that
  $\HH^\C$ is simple. Note that, up to isomorphism, $\so(5,\C)$ is the
  only simple complex Lie algebra of complex dimension $10$. From this
  we conclude that $\HH$ is isomorphic to either $\so(4,1)$ or
  $\so(3,2)$ (the only non-compact real forms of $\so(5,\C)$ up to
  isomorphism) thus showing that (3) holds when $\HH$ is simple and
  $(p,q) = (3,1)$.
\end{proof}

\section{Proof of the Main Theorem}
\label{section-structure-tildeM}
\noindent
In this section we will assume the hypotheses of the Main
Theorem. More precisely, we assume that $M$ is a connected analytic
pseudo-Riemannian manifold which is complete weakly irreducible and
has finite volume. We also assume that $M$ admits an analytic and
isometric $\uSO_0(p,q)$-action with a dense orbit for some integers
$p,q$ such that $p,q \geq 1$ and $n = p+q\ge 5$, or $(p,q) =
(3,1)$. In the case $(p,q) = (3,1)$ we assume that $X^*\perp Y^*$ on
$M$ for every $X \in \su(2)$ and $Y \in
i\su(2)$.
Finally we are assuming that $\dim(M) = n(n+1)/2$ and we will consider
the three cases provided by Lemma~\ref{lemma-HH-Lie-structure}. For
this we will use the notation from
Section~\ref{section-structure-centralizer}. Our first goal is to rule
out cases (1) and (2) of Lemma~\ref{lemma-HH-Lie-structure}, which is
done in the next two subsections. Then, we obtain in the third
subsection the conclusions of the Main Theorem when case (3) of
Lemma~\ref{lemma-HH-Lie-structure} holds.

\subsection{Case 1: The radical of $\HH$ is non-trivial}
We are assuming that the conclusion (1) of
Lemma~\ref{lemma-HH-Lie-structure} is satisfied for some fixed $x_0
\in \widetilde{M}$. We will see that this yields a contradiction with
our assumptions on $M$.

First note that $\HH(x_0) = \GG(x_0) \oplus \VV(x_0)$ is a Lie
subalgebra of $\HH$. This holds because $\VV(x_0)$ is Abelian and
$[\GG(x_0), \VV(x_0)] \subset \VV(x_0)$. Hence,
$\GG(x_0)\oplus\VV(x_0)$ is isomorphic to the semidirect product Lie
algebra $\so(p,q) \ltimes W$, where $W$ is an $n$-dimensional
$\so(p,q)$-module endowed with the Abelian Lie algebra
structure. Choose a Lie algebra isomorphism $\psi : \so(p,q)\ltimes W
\rightarrow \HH(x_0)$ that maps $\so(p,q)$ onto $\GG(x_0)$ and $W$
onto $\VV(x_0)$.

Let us denote with $\uSO_0(p,q) \ltimes W$ the Lie group structure on
$\uSO_0(p,q)\times W$ with the semidirect product given by:
\[
  (A,v)\cdot(B,w) = (AB, B^{-1}v + w),
\]
where we are considering the representation of $\uSO_0(p,q)$ on $W$
induced by that of $\so(p,q)$. In particular, the Lie algebra of
$\uSO_0(p,q) \ltimes W$ is $\so(p,q) \ltimes W$. By
Lemma~\ref{lemma-Kill-to-action}, there exists an analytic isometric
right action of $\uSO_0(p,q) \ltimes W$ on $\widetilde{M}$ such that
$\psi(X) = X^*$ for every $X\in \so(p,q)\ltimes W$. Since $\HH$
centralizes the left $\uSO_0(p,q)$-action, then the right $\uSO_0(p,q)
\ltimes W$-action centralizes the left $\uSO_0(p,q)$-action as well
and so it preserves both $T\OO$ and $T\OO^\perp$.

Using the right $(\uSO_0(p,q) \ltimes W)$-action, let us now consider
the map:
\[
  f : \uSO_0(p,q) \ltimes W \rightarrow \widetilde{M}, \quad
  h \mapsto x_0 h,
\]
which is clearly $(\uSO_0(p,q) \ltimes W)$-equivariant for the right
action on its domain. In what follows, we will denote with $I$ the
identity element in $\uSO_0(p,q)$. Then, $df_{(I,0)}$ is the
composition:
\begin{alignat*}{3}
  \so(p,q) \ltimes W &\rightarrow \GG(x_0)\oplus\VV(x_0)
  \rightarrow T_{x_0}\widetilde{M} \\
  X &\mapsto X^* \mapsto X^*_{x_0}.
\end{alignat*}
Hence by the property $\psi(X) = X^*$ ($X\in \so(p,q)\ltimes W$) and
Lemma~\ref{lemma-HH-decomposition}, $df_{(I,0)}$ maps $\so(p,q)$ onto
$T_{x_0}\OO$ and $W$ onto $T_{x_0}\OO^\perp$. In particular, $f$ is a
local diffeomorphism at $(I,0)$.

For every $w\in W$, denote with $R_w$ the transformations on both
$\uSO_0(p,q)\ltimes W$ and $\widetilde{M}$ given by the assignment $x
\mapsto x(I,w)$. In particular, a straightforward computation shows
that we have:
\begin{alignat*}{2}
  d(R_w)_{(I,v)} : T_{(I,v)}(\uSO_0(p,q)\ltimes W) &\rightarrow
  T_{(I,v+w)}(\uSO_0(p,q)\ltimes W) \\
  (X,Y_v) &\mapsto (X,Y_{v+w}).
\end{alignat*}
Also note that $R_w(I\times W) = I\times W$, since $W$ is a subgroup
of $\uSO_0(p,q)\ltimes W$.

Let $N = f(I\times W)$, which defines a submanifold of $\widetilde{M}$
in a neighborhood of $x_0 = f(I,0)$. By the above remarks on
$df_{(I,0)}$ we have:
\[
  T_{f(I,0)} N = df_{(I,0)}(T_{(I,0)}(I\times W)) =
  T_{f(I,0)}\OO^\perp.
\]
But then, the equivariance of $f$ yields:
\begin{alignat*}{2}
  T_{f(I,w)} N &= df_{(I,w)}(T_{(I,w)}(I\times W))
  = df_{(I,w)}(d(R_w)_{(I,0)}(T_{(I,0)}(I\times W))) \\
  &= d(R_w\circ f)_{(I,0)}(T_{(I,0)}(I\times W)))
  = d(R_w)_{f(I,0)}(T_{f(I,0)} N) \\
  &= d(R_w)_{f(I,0)}(T_{f(I,0)} \OO^\perp) = T_{R_w(f(I,0))} \OO^\perp
  = T_{f(I,w)} \OO^\perp,
\end{alignat*}
where we have used in the second to last identity that $R_w$ preserves
in $\widetilde{M}$ the bundle $T\OO^\perp$. This proves that $N$ is an
integral submanifold of $T\OO^\perp$ passing through the point $x_0 =
f(I,0)$.

On the other hand, from the left $\uSO_0(p,q)$-action on
$\widetilde{M}$ we obtain by restriction to $N$ a map:
\[
\varphi: \uSO_0(p,q)\times N \rightarrow \widetilde{M}, \quad (g,x)
\mapsto gx,
\]
whose differential at $(I,x_0)$ is given by:
\[
  X + v \mapsto X^*_{x_0} + v,
\]
where $X \in \so(p,q)$ and $v \in T_{x_0}N$. The latter is an
isomorphism and so the map $\varphi$ is a diffeomorphism from a
neighborhood of $(I,x_0)$ onto a neighborhood of $x_0$. Since the left
$\uSO_0(p,q)$-action preserves both $T\OO$ and $T\OO^\perp$, we
conclude that there is an integral submanifold of $T\OO^\perp$ passing
through every point in neighborhood of $x_0$ in $\widetilde{M}$.
Hence, the tensor $\Omega$ considered in Lemma~\ref{lemma-omega-Omega}
vanishes in a neighborhood of $x_0$. Since all our objects are
analytic, this implies that $\Omega$ vanishes everywhere and so
Lemma~\ref{lemma-omega-Omega} implies the integrability of
$T\OO^\perp$ everywhere in $\widetilde{M}$.

This last conclusion and Proposition~\ref{prop-int-normal} contradict
the weak irreducibility of $M$. This shows that case (1) from
Lemma~\ref{lemma-HH-Lie-structure} cannot occur.

\subsection{Case 2: $\HH_0(x_0) \not= \{0\}$ and $\HH$ is the sum of
  two simple ideals}
We now assume that the conclusion (2) of
Lemma~\ref{lemma-HH-Lie-structure} is satisfied for some fixed $x_0
\in \widetilde{M}$. As in the previous case, we will rule out this
possibility.

In this case, there exist simple Lie algebras $\h_1$, $\h_2$ and an
isomorphism of Lie algebras $\psi : \h_1 \times \h_2 \rightarrow \HH$
so that $\psi(\h_2) = \HH_0(x_0)\oplus\VV(x_0)$. Let $H_1$ and $H_2$
be simply connected Lie groups with Lie algebras $\h_1$ and $\h_2$,
respectively. By Lemma~\ref{lemma-Kill-to-action}, there is an
analytic isometric right action of $H_1 \times H_2$ on $\widetilde{M}$
such that $\psi(X) = X^*$ for every $X \in \h_1 \times \h_2$. Note
that this right action centralizes the left $\uSO_0(p,q)$-action on
$\widetilde{M}$ and so it preserves the bundles $T\OO$ and
$T\OO^\perp$.

Let us consider the map:
\[
  f : H_1 \times H_2 \rightarrow \widetilde{M}, \quad
  h \mapsto x_0h,
\]
which is clearly $(H_1\times H_2)$-equivariant for the right action on
its domain. In particular, we have $df_e(X) = X^*_{x_0} =
\psi(X)_{x_0}$ which is surjective with $\ker(df_e) =
\psi^{-1}(\HH_0(x_0))$ by Lemma~\ref{lemma-HH-decomposition}. We claim
that we also have $df_e(\h_1) = T_{x_0}\OO$ and $df_e(\h_2) =
T_{x_0}\OO^\perp$. The latter follows from our choice of $\psi$ and
Lemma~\ref{lemma-HH-decomposition}. To prove the former, first note
that both $\psi(\h_1)$ and $\GG(x_0)$ are complementary
$\so(p,q)$-modules to $\HH_0(x_0)\oplus\VV(x_0)$ in $\HH$, and so they
are isomorphic as $\so(p,q)$-modules. Hence, by
Lemma~\ref{lemma-HH-decomposition} the evaluation $ev_{x_0} : \HH
\rightarrow T_{x_0}\widetilde{M} = T_{x_0}\OO \oplus T_{x_0}\OO^\perp$
necessarily maps $\psi(\h_1)$ onto $T_{x_0}\OO$ since $T_{x_0}\OO
\simeq \GG(x_0) \not\simeq T_{x_0}\OO^\perp$ as
$\so(p,q)$-modules. This proves that $df_e(\h_1) = T_{x_0}\OO$.

Let us denote with $H$ the connected subgroup of $H_2$ with Lie
algebra $\psi^{-1}(\HH_0(x_0))$, the latter being isomorphic to
$\so(p,q)$ by Lemma~\ref{lemma-HH-decomposition}. Since $H_2$ is
simply connected and $\psi^{-1}(\HH_0(x_0))$ is simple, it follows
that $H$ is a closed subgroup of $H_2$ (see Exercise~D.4(ii), Chapter~II
from \cite{Helgason}). Hence, the map:
\[
  \widehat{f} : H_1 \times H\backslash H_2 \rightarrow
  \widetilde{M}, \quad
  (h_1,H h_2) \mapsto x_0 (h_1,h_2),
\]
is a well-defined $(H_1 \times H_2)$-equivariant analytic map of
manifolds. From the properties of $df_e$, it also follows that
$\widehat{f}$ is a local diffeomorphism at $(e_1,H e_2)$.

By considering $N = \widehat{f}(I \times H\backslash H_2)$ and using
the equivariance of $\widehat{f}$, we can prove with arguments similar
to those used in the previous subsection that $T\OO^\perp$ is
integrable. This again rules out the current case.

\subsection{Case 3: $\HH$ is simple}
In this case we are now assuming that (3) from
Lemma~\ref{lemma-HH-Lie-structure} holds for some $x_0 \in
\widetilde{M}$ that we now consider fixed.

\begin{lemma}\label{lemma-HH-Tc}
  There is an isomorphism $\psi : \so(p,q)\oplus\R^{p,q} \rightarrow
  \HH = \GG(x_0)\oplus\VV(x_0)$ of Lie algebras that preserves the
  summands in that order, where the domain has the Lie algebra
  structure given by $[\cdot,\cdot]_c$ for some $c \not= 0$ as defined
  in Lemma~\ref{lemma-Lie-so(p,q)Rpq}. In particular, $\psi$ is an
  isomorphism of $\so(p,q)$-modules as well.
\end{lemma}
\begin{proof}
  The result follows from the arguments in the second to last
  paragraph in the proof of Lemma~\ref{lemma-HH-Lie-structure} when
  $\VV(x_0) \simeq \R^{p,q}$ as $\so(p,q)$-modules and $n \geq 5$.
  Hence, by Lemma~\ref{lemma-TxOperp-is-Rpq} we can assume that either
  $(p,q) = (3,1)$ or $(p,q) = (4,4)$ in the rest of the proof.

  By Lemma~\ref{lemma-HH-Lie-structure} there is an isomorphism $\psi
  : \h \rightarrow \HH$, for $\h = \so(V)$ where $V$ is either
  $\R^{4,1}$ or $\R^{3,2}$ for $(p,q) = (3,1)$, and it is $\R^{5,4}$
  for $(p,q) = (4,4)$. The restriction of this homomorphism to
  $\psi^{-1}(\GG(x_0))$ yields a representation of $\GG(x_0) \simeq
  \so(p,q)$ on the $(n+1)$-dimensional space $V$. By the description
  of the irreducible representations of $\so(3,1)$ from previous
  sections we know that there do not exist $5$-dimensional irreducible
  representations of $\so(3,1)$. In particular, there is a line $L
  \subset V$ that is a $\GG(x_0)$-submodule for the case $(p,q) =
  (3,1)$. On the other hand, since $\so(4,4)$ is split and using
  Weyl's dimension formula we find that $\so(4,4)$ does not admit
  $9$-dimensional irreducible representations. Hence, for the case
  $(p,q) = (4,4)$ we similarly conclude the existence of a line $L
  \subset V$ which is a $\GG(x_0)$-submodule.

  Let us now consider our two remaining cases $(p,q) \in \{(3,1),
  (4,4)\}$ together. If $L$ as above is a null line, then
  $\psi^{-1}(\GG(x_0))$ lies in the Lie algebra $\mathfrak{s}$ of the
  stabilizer of some point in either of the pseudo-conformal spaces
  $C^{p,q-1}$ or $C^{p-1,q}$. We recall that $C^{r,s}$ is the
  projectivization of the null cone of $\R^{r+1,s+1}$, is homogeneous
  under $\mathrm{O}(r+1,s+1)$ and has dimension $r+s$ (see
  \cite{Akivis} for further details). In particular, $\mathfrak{s}$
  has dimension $n(n-1)/2 + 1$. This yields $\psi^{-1}(\GG(x_0))
  \subsetneq \mathfrak{s} \subsetneq \h$, which contradicts
  Theorem~\ref{co-Max}. We conclude that $L$ is a non-null line.

  This yields an orthogonal decomposition $V = L \oplus L^\perp$
  into non-degenerate subspaces which is clearly a decomposition into
  $\GG(x_0)$-submodules. Hence, $\psi$ induces an isomorphism
  $\so(L^\perp) \rightarrow \GG(x_0)$ and a rank argument shows that
  $L^\perp$ has signature $(p,q)$.
  In particular, $\so(L^\perp) \simeq \so(p,q)$ as Lie algebras under
  $\psi$. With respect to the corresponding $\so(p,q)$-module
  structure, it is easily seen that $\so(L^\perp)$ has a complementary
  module in $\h$ isomorphic to $\R^{p,q}$. This provides an
  isomorphism $\h \simeq \so(p,q) \oplus \R^{p,q}$ so that the Lie
  algebra structure on $\h$ corresponds to the one given by
  $[\cdot,\cdot]_c$ on $\so(p,q) \oplus \R^{p,q}$ for some $c \not=
  0$. Hence, under the identification $\h \simeq \so(p,q) \oplus
  \R^{p,q}$ of Lie algebras, $\psi$ is the required isomorphism.  
\end{proof}

Let us fix an isomorphism of Lie algebras $\psi : \so(p,q)
\oplus\R^{p,q} \rightarrow \HH = \GG(x_0)\oplus\VV(x_0)$ as in
Lemma~\ref{lemma-HH-Tc}. We will identify $\h =
\so(p,q)\oplus\R^{p,q}$ with either $\so(p+1,q)$ or $\so(p,q+1)$
through the appropriate isomorphism as considered in
Lemma~\ref{lemma-Lie-so(p,q)Rpq}. Also, we will denote with $H$ either
$\uSO_0(p+1,q)$ or $\uSO_0(p,q+1)$, chosen so that $\mathrm{Lie}(H) =
\h$.

By Lemma~\ref{lemma-Kill-to-action}, there is an analytic isometric
right $H$-action on $\widetilde{M}$ such that $\psi(X) = X^*$ for
every $X \in \h$. As in the previous subsections, we now consider the
orbit map:
\[
  f : H \rightarrow \widetilde{M}, \quad
  h \mapsto x_0 h
\]
which satisfies $df_I(X) = X^*_{x_0} = \psi(X)_{x_0}$ for every $X \in
\h$. By the choice of $\psi$ and Lemma~\ref{lemma-HH-decomposition} it
follows that $df_I$ is an isomorphism that maps $\so(p,q)$ onto
$T_{x_0}\OO$ and $\R^{p,q}$ onto $T_{x_0}\OO^\perp$. Since $f$ is
$H$-equivariant for the right action on its domain, we conclude that
it is an analytic local diffeomorphism.

\begin{lemma}\label{lemma-pullbackmetric}
  Let $\overline{g}$ be the metric on $\h = \so(p,q) \oplus \R^{p,q}$
  defined as the pullback under $df_I$ of the metric $g_{x_0}$ on
  $T_{x_0} \widetilde{M}$. Then, $\overline{g}$ is
  $\so(p,q)$-invariant.
\end{lemma}
\begin{proof}
  By the above expression of $df_I$ and since $\psi$ is an isomorphism
  of Lie algebras with $\psi(\so(p,q)) = \GG(x_0)$, it is enough to
  show that the metric on $\HH$ defined as the pullback of $g_{x_0}$
  with respect to the evaluation map:
  \[
    \HH \rightarrow T_{x_0} \widetilde{M}, \quad
    X \mapsto X_{x_0}
  \]
  is $\GG(x_0)$-invariant. For simplicity, we will denote with
  $\overline{g}$ such metric on $\HH$. Let $X,Y,Z \in \HH$ be given
  with $X \in \GG(x_0)$. In particular, there exist $X_0 \in \so(p,q)$
  such that $X = \rho_{x_0}(X_0) + X_0^*$, where $\rho_{x_0}$ is the
  homomorphism from Proposition~\ref{prop-g(x)} and $X_0^*$ is the
  vector field on $\widetilde{M}$ induced by $X_0$ through the left
  $\uSO_0(p,q)$-action. Then, the following proves the required
  invariance:
  \begin{align*}
    \overline{g}([X,Y], Z)
    &= g_{x_0}([X,Y]_{x_0},Z_{x_0}) = g([X,Y],Z)|_{x_0} \\
    &= g([\rho_{x_0}(X_0) + X_0^*,Y],Z)|_{x_0} = g([\rho_{x_0}(X_0),Y],Z)|_{x_0} \\
    &= \rho_{x_0}(X_0)(g(Y,Z))|_{x_0} - g(Y,[\rho_{x_0}(X_0),Z])|_{x_0} \\
    &= -g(Y,[\rho_{x_0}(X_0),Z])|_{x_0} = -g(Y,[\rho_{x_0}(X_0) + X_0^*,Z])|_{x_0} \\
    &= -g(Y,[X,Z])|_{x_0} = -\overline{g}(Y,[X,Z]).
  \end{align*}
  We have used in lines 2 and 4 that $\HH$ centralizes $X_0^*$. To
  obtain the third line we used that $\rho_{x_0}(X_0)$ is a Killing
  field for the metric $g$. And the first identity in line 4 uses the
  fact that $\rho_{x_0}(X_0)$ vanishes at $x_0$.
\end{proof}

{}From the previous result and Lemma~\ref{lemma-inner-so(p,q)Rpq}, for
$n \geq 5$ we can rescale the metric along the bundles $T\OO$ and
$T\OO^\perp$ in $M$ so that the new metric $\widehat{g}$ on
$\widetilde{M}$ satisfies $(df_I)^*(\widehat{g}_{x_0}) = K$, the
Killing form on $\h$.

Let us now consider the case $(p,q) = (3,1)$. From the hypotheses of
the Main Theorem, we are now assuming that $X^* \perp Y^*$ in
$\widetilde{M}$ for every $X \in \su(2), Y \in J\su(2)$ and the left
$\SL(2,\C)$-action on $\widetilde{M}$. By Proposition~\ref{prop-g(x)}
and Lemma~\ref{lemma-HH-module-struc}, we have
$\widehat{\rho}_{x_0}(X)_{x_0} = X^*_{x_0}$ and so
$\widehat{\rho}_{x_0}(X)_{x_0}\perp \widehat{\rho}_{x_0}(Y)_{x_0}$,
for every $X \in \su(2)$ and $Y \in J\su(2)$. Hence, for the compact
real form $\mathcal{U} = \widehat{\rho}_{x_0}(\su(2))$ of $\GG(x_0)$
we have $X_{x_0} \perp Y_{x_0}$ when $X \in \mathcal{U}$ and $Y \in
J\mathcal{U}$. By the definition of $\overline{g}$ and the expression
for $df_I$ given above, it follows that for the metric $\overline{g}$
restricted to $\so(3,1) \simeq \gsl(2,\C)_\R$ we have
$\overline{g}(X,Y) = 0$ for every $X \in \psi^{-1}(\mathcal{U})$ and
$Y \in J\psi^{-1}(\mathcal{U})$. By the remarks that follow
Lemma~\ref{lemma-inner-so(p,q)Rpq} we conclude that in this case we
can also rescale the metric on $\widetilde{M}$ along $T\OO$ and
$T\OO^\perp$ to obtain a new metric $\widehat{g}$ such that
$(df_I)^*(\widehat{g}_{x_0}) = K$, the Killing form of $\h$.

Note that, for both cases $n \geq 5$ and $(p,q) = (3,1)$, since the
elements of $\HH$ preserve the decomposition $TM = T\OO\oplus
T\OO^\perp$, then $\HH \subset \Kill(\widetilde{M},\widehat{g})$. In
other words, the elements of $\HH$ are Killing vector fields for the
metric $\widehat{g}$ rescaled as above. In particular, $\widehat{g}$
is invariant under the right $H$-action. Similarly, the left
$\uSO_0(p,q)$-action on $\widetilde{M}$, from the hypotheses of the
Main Theorem, preserves the rescaled metric $\widehat{g}$. Also note
that the metric $\widehat{g}$ is the lift of a correspondingly
rescaled metric $\widehat{g}$ in $M$.

Consider the bi-invariant metric on $H$ induced by the Killing form
$K$, which we will denote with the same symbol. The previous
discussion implies that the local diffeomorphism $f : (H, K)
\rightarrow (\widetilde{M},\widehat{g})$ is a local isometry. Then, by
Corollary~29 in page 202 of \cite{ONeill-book}, the completeness of
$(H, K)$ and the simple connectedness of $\widetilde{M}$ imply that
$f$ is an isometry.

Hence, from the previous discussion we obtain the following result.

\begin{lemma}
\label{lemma-structure-tildeM}
Let $M$ be as in the statement of the Main Theorem. If $\dim(M) =
n(n+1)/2$, then for $H$ either $\uSO_0(p,q+1)$ or $\uSO_0(p+1,q)$,
there exists an analytic diffeomorphism $f : H \rightarrow
\widetilde{M}$ and an analytic isometric right $H$-action on
$\widetilde{M}$ such that:
  \begin{enumerate}
  \item on $\widetilde{M}$, the left $\uSO_0(p,q)$-action and the right
    $H$-action commute with each other,
  \item $f$ is $H$-equivariant for the right $H$-action on its domain,
  \item for a pseudo-Riemannian metric $\widehat{g}$ in $M$ obtained
    by rescaling the original one on the summands of $TM = T\OO\oplus
    T\OO^\perp$, the map $f : (H,K) \rightarrow
    (\widetilde{M},\widehat{g})$ is an isometry where $K$ is the
    bi-invariant metric on $H$ induced from the Killing form of its
    Lie algebra.
  \end{enumerate}
\end{lemma}

If we consider $H$ endowed with the bi-invariant pseudo-Riemannian
metric $K$ induced by the Killing form of its Lie algebra, then
Lemma~\ref{lemma-structure-tildeM} allows to consider $(H,K)$ as the
isometric universal covering space of $(M,\widehat{g})$. We will use
this identification in the rest of the arguments.

The isometry group $\Iso(H)$ for the pseudo-Riemannian manifold
$(H,K)$ has finitely many connected components (see for example
Section~4 of \cite{Quiroga-Annals}). Furthermore, the connected
component of the identity is given as $\Iso_0(H) = L(H)R(H)$, the
subgroup generated by $L(H)$ and $R(H)$, the left and right
translations, respectively.

Let $\rho : \uSO_0(p,q) \rightarrow \Iso_0(H)$ be the homomorphism
induced by isometric left $\uSO_0(p,q)$-action on $H$.  With respect
to the natural covering $H\times H \rightarrow L(H)R(H)$, this yields
homomorphisms $\rho_1, \rho_2 : \uSO_0(p,q) \rightarrow H$ such that:
\[
  \rho(g) = L_{\rho_1(g)}\circ R_{\rho_2(g)^{-1}},
\]
for every $g \in \uSO_0(p,q)$. By Lemma~\ref{lemma-structure-tildeM}
we have $\rho(g)\circ R_h = R_h \circ \rho(g)$ for every $g \in
\uSO_0(p,q)$ and $h \in H$. In particular, $\rho_2(\uSO_0(p,q))$ lies
in the center $Z(H)$ and so (being connected) it is trivial. We
conclude that $\rho = L_{\rho_1}$ which implies that the
$\uSO_0(p,q)$-action on $H$ is induced by the homomorphism $\rho_1 :
\uSO_0(p,q) \rightarrow H$ and the left action of $H$ on itself.  Note
that $\rho_1$ is necessarily non-trivial.

By Lemma~\ref{lemma-structure-tildeM}, we have $\pi_1(M) \subset
\Iso(H)$, and from the above remarks $\Gamma_1 =
\pi_1(M)\cap\Iso_0(H)$ is a finite index subgroup of $\pi_1(M)$. In
particular, every $\gamma \in \Gamma_1$ can be written as $\gamma =
L_{h_1}\circ R_{h_2}$ for some $h_1, h_2 \in H$.

On the other hand, since the left $\uSO_0(p,q)$-action on $H$ is the
lift of an action on $M$, it follows that it commutes with the
$\Gamma_1$-action. Applying this property to $\gamma = L_{h_1}\circ
R_{h_2}$ we conclude that $L_{h_1}\circ L_{\rho_1(g)} = L_{\rho_1(g)}
\circ L_{h_1}$, for every $g \in \uSO_0(p,q)$, which implies $\Gamma_1
\subset L(Z)R(H)$, where $Z$ is the centralizer of
$\rho_1(\uSO_0(p,q))$ in $H$. By Lemma~\ref{lemma-cent-homomorphism},
the center of $Z(H)$ has finite index in $Z$, which implies that
$R(H)$ has finite index in $L(Z)R(H)$. In particular, $\Gamma =
\Gamma_1 \cap R(H)$ is a finite index subgroup of $\Gamma_1$, and so
has finite index in $\pi_1(M)$ as well.

Hence, the natural identification $R(H) = H$ realizes $\Gamma$ as a
discrete subgroup of $H$ such that $H/\Gamma$ is a finite covering
space of $M$. Furthermore, if $\varphi : H/\Gamma \rightarrow M$ is
the corresponding covering map, and for the left $\uSO_0(p,q)$-action
on $H/\Gamma$ given by the homomorphism $\rho_1 : \uSO_0(p,q)
\rightarrow H$, then the above constructions show that $\varphi$ is
$\uSO_0(p,q)$-equivariant. We also note that $\varphi$ is a local
isometry for the metric $\widehat{g}$ on $M$ considered in
Lemma~\ref{lemma-structure-tildeM}.

To complete the proof of the Main Theorem it only remains to show that
$\Gamma$ is a lattice in $H$. For this it is enough to prove that $M$
has finite volume in the metric $\widehat{g}$.  The following result
provides proofs of these facts since we are assuming that $M$ has
finite volume in its original metric.

\begin{lemma}
  Let us denote with $\vol$ and $\vol_{\widehat{g}}$ the volume
  elements on $M$ for the original metric on $M$ and the rescaled
  metric $\widehat{g}$, respectively. Then, there is some constant $C
  > 0$ such that $\vol_{\widehat{g}} = C \vol$.
\end{lemma}
\begin{proof}
  Clearly, it suffices to verify this locally, so we consider some
  coordinates $(x^1, \dots, x^m)$ of $M$ in a neighborhood $U$ of a
  given point such that $(x^1, \dots, x^r)$ defines a set of
  coordinates of the leaves of the foliation $\OO$ in such
  neighborhood. For the original metric $g$ on $M$, consider as above
  the orthogonal bundle $T\OO^\perp$ and a set of $1$-forms $\theta^1,
  \dots \theta^{m-r}$ that define a basis for its dual
  $(T\OO^\perp)^*$ at every point in $U$. Hence, in $U$ the metric $g$
  has an expression of the form:
  \[
  g = \sum_{i,j = 1}^r h_{ij} dx^i\otimes dx^j
  + \sum_{i,j = 1}^{m-r} k_{ij} \theta^i\otimes \theta^j.
  \]
  {}From this and the definition of the volume element as an $m$-form,
  its is easy to see that:
  \[
  \vol = \sqrt{|\det(h_{ij})\det(k_{ij})|} dx^1\wedge \dots \wedge
  dx^r \wedge \theta^1 \wedge \dots \wedge \theta^{m-r}.
  \]
  On the other hand, the metric $\widehat{g}$ is obtained by rescaling
  $g$ along the bundles $T\OO$ and $T\OO^\perp$, and so it has an
  expression of the form:
  \[
  \widehat{g} = \sum_{i,j = 1}^r c_1 h_{ij} dx^i\otimes dx^j
  + \sum_{i,j = 1}^{m-r} c_2 k_{ij} \theta^i\otimes \theta^j,
  \]
  for some constants $c_1, c_2 \neq 0$. Hence, the volume element of
  $\widehat{g}$ satisfies:
  \begin{alignat*}{2}
    \vol_{\widehat{g}} &= \sqrt{|\det(c_1h_{ij})\det(c_2k_{ij})|} dx^1\wedge
    \dots \wedge dx^r \wedge \theta^1 \wedge \dots \wedge
    \theta^{m-r}\\
    &= \sqrt{|c_1^r c_2^{m-r}|} \vol.
  \end{alignat*}
\end{proof}

\begin{appendices}
\appendix
\section{Facts on the Lie algebra $\so (p,q)$}
\label{section-Lie-algebras}
\noindent
In this appendix we collect some facts about the pairs $(\so
(p+1,q),\so (p,q))$ and $(\so (p,q+1),\so (p,q))$.  We now describe
some low-dimensional non-trivial $\so(p,q)$-modules. The following
result is an easy consequence of well known facts of classical Lie
groups (see \cite{GoodmanWallach,Onishchik}). As mentioned in the
introduction we denote with $m(\g)$ the lowest dimension of a
non-trivial $\g$-module with an invariant non-degenerate symmetric
bilinear form.

\begin{lemma}\label{lemma-Rpq-smallest}
  Let $p,q \geq 1$, $n = p + q \geq 4$ such that $(p,q) \not=
  (2,2)$. Then, $m(\so(p,q)) = n$, i.e.~there is no non-trivial
  $\so(p,q)$-module with dimension $< n$ and carrying an invariant
  inner product. Moreover, the only non-trivial irreducible
  $\so(p,q)$-module of dimension $\leq n$ is $\R^{p,q}$, except for
  the Lie algebras $\so(3,1)$, $\so(3,2)$, $\so(3,3)$ and
  $\so(4,4)$. For the latter, there also exist the following
  irreducible modules:
  \begin{itemize}
  \item $\C^2_\R$ corresponding to $\so(3,1) \simeq \gsl(2,\C)_\R$.
  \item $\R^4$ corresponding to $\so(3,2) \simeq \symp(2,\R)$.
  \item $\R^4$ and $\R^{4*}$ corresponding to $\so(3,3) \simeq
    \gsl(4,\R)$.
  \item $\so(4,4)$-invariant real forms of the half spin
    representations of $\so(8,\C)$, both $8$-dimensional.
  \end{itemize}
\end{lemma}

\begin{remark}\label{remark1}
  We note that the previous statement is not correct in lower
  dimensions. We have that $\so (1,1)$ is abelian, so every
  irreducible module is one-dimensional. In fact $\R^{1,1}=\R^2$
  decomposes into irreducible submodules where $t\in \R\simeq \so
  (1,1)$ acts by multiplication by $e^t$ respectively $e^{-t}$.

  The Lie algebra $\so (2,2)\simeq \gsl (2,\R)\times \gsl (2,\R)$ is
  not simple and acts on $\R^2$. But there is again no non-trivial
  invariant symmetric bilinear form on $\R^2$. The lowest dimensional
  $\so (2,2)$-module with a non-trivial invariant form is $\R^{2,1}$
  where one of the factors carries the canonical $\so (2,1)$ action,
  and the other factor acts trivially, thus reducing this to the case
  of $\so (2,1)$. Hence we have $m(\so (2,2)) = 3$.
\end{remark}

The following statement is an easy to prove exercise.

\begin{lemma}\label{lemma-Rpq-so(p,q)}
  Let $p,q \geq 1$ and $n = p+q \geq 3$. Then, for every $c \in \R$,
  the map $T_c : \wedge^2\R^{p,q} \rightarrow \so(p,q)$ given by:
  \[
  T_c(u\wedge v) = c\left<u,\cdot\right>_{p,q}v -
  c\left<v,\cdot\right>_{p,q}u,
  \]
  for every $u,v \in \R^{p,q}$, is a well defined homomorphism of
  $\so(p,q)$-modules. Also, $T_c$ is an isomorphism of
  $\so(p,q)$-modules if and only if $c\neq 0$. If $n \not= 4$, then
  these maps exhaust all the $\so(p,q)$-module homomorphisms
  $\wedge^2\R^{p,q} \rightarrow \so(p,q)$.
\end{lemma}

\begin{remark}\label{remark2}
  The last conclusion fails for $\so (2,2)\simeq \so (2,1)\times \so
  (2,1)$ as it is not simple. It also fails for $\so (3,1)\simeq \gsl
  (2,\C)$, since the complex structure of $\gsl(2,\C)$ defines an
  isomorphism which is not multiplication by a real scalar.
\end{remark}

Next we prove the maximality of $\so(p,q)$ in both $\so(p,q+1)$ and
$\so(p+1,q)$.

\begin{theorem}\label{co-Max}
  Assume that $p,q \geq 1$ and $n = p + q \ge 3$, and let $\g=\so
  (p+1,q)$ or $\g=\so (p,q+1)$. Suppose that $\rho :\so
  (p,q)\hookrightarrow \g$ is an injective Lie algebra homomorphism
  and let $\h =\rho (\so (p,q))$.  If $\g ,\h \not\simeq \so
  (2,1)\times \so (2,1)$, then $\h$ is a maximal subalgebra of $\g$.
\end{theorem} 
\begin{proof}
  This follows from Theorem 1.2 of \cite{D52} for $n$ big and a case
  by case calculation for the other cases. We give here a simple proof
  for completeness. The map $\theta (X)=-X^t$ is a Cartan involution
  on $\g$. As all Cartan involutions are conjugate and every Cartan
  involution on $\h$ extends to a Cartan involution on $\g$ we can
  assume that $\theta (\h) =\h$. Then the form $\beta (X ,Y)
  =\mathrm{Tr} (XY)$ is non-degenerate on $\g$ and $\h$. Let $V$ be
  the $\beta$-orthogonal complement of $\h$. Hence $V$ is an
  $n$-dimensional $\h$-module which is necessarily non-trivial since
  otherwise $\h$ is an ideal. Furthermore, $\beta|_{V\times V}$ is
  non-degenerate and $\h$-invariant. It follows by
  Lemma~\ref{lemma-Rpq-smallest} that $V$ is an irreducible
  $\h$-module if $n\ge 4$ and $(p,q)\not=(2,2), (3,2), (3,3)$. Hence
  $\h$ is maximal in those cases.

  Let us now assume that $\h=\so (2,1)\simeq \gsl (2,\R)$, $\g =\so
  (3,1)$ and $V$ is not irreducible.  Thus $V=V_1\oplus V_2$ with
  $V_1\simeq \R$ and $V_2\simeq \R^2$ where the representation on
  $V_1$ is trivial and the representation on $V_2$ is isomorphic to
  the natural representation of $\gsl (2,\R)$. But then it follows
  that $V_2$ is invariant under $\theta$ and hence $\beta|_{V_2\times
    V_2}$ is an $\h$-invariant non-degenerate form, contradicting the
  fact that there is no such form on $\R^2$. The remaining cases
  $(p,q) \in \{ (3,2), (3,3) \}$ can be considered similarly using
  Lemma~\ref{lemma-Rpq-smallest}.
\end{proof}

\begin{lemma}\label{lemma-cent-homomorphism}
  Suppose that $G$ is a connected Lie group locally isomorphic to
  either $\uSO_0(p,q+1)$ or $\uSO_0(p+1,q)$, where $p,q\geq 1$ and $n
  = p+q \ge 3$, and consider $\rho : \uSO_0(p,q) \rightarrow G$ a
  non-trivial homomorphism of Lie groups. Assume that $\so (p,q)$, $\so
  (p,q+1)$ and $\so (p+1,q)$ satisfy the same conditions as in
  Theorem \ref{co-Max}. Then, the centralizer $Z_G(\rho(\uSO_0(p,q)))$
  of $\rho(\uSO_0(p,q))$ in $G$ contains $Z(G)$ {\rm(}the center of
  $G${\rm)} as a finite index subgroup.
\end{lemma}
\begin{proof} Write $H=\rho(\uSO_0(p,q))$ and write $\h $ for the Lie
  algebra of $H$. Then, clearly $Z(G)\subseteq Z_G(H)$. The Lie
  algebra of $Z_G(H)$ is
  \[
  \mathfrak{z}_{\g}(\h)=\{X\in\g \mid [X,Y]=0 \mbox{ for all } Y \in
  \h\}\, .
  \]
  Clearly $\mathfrak{z}_{\g}(\h)+\h$ is a Lie algebra containing
  $\h$. By Theorem~\ref{co-Max} we conclude that
  $\mathfrak{z}_{\g}(\h)\subset \h$. As $\h$ is simple, it follows
  that $\mathfrak{z}_{\g}(\h)=\{0\}$ and so $Z_G(H)$ is discrete.
  Finally, it follows easily from \cite{Warner72}, Lemma~1.1.3.7, that
  $Z_G(H)$ is contained in any maximal compact subgroup of $G$ and
  hence that it is finite.
\end{proof}

\begin{remark}
  If $p=q=1$ then $\h$ is abelian, hence $H\subseteq Z_G(H)$. On the
  other hand Theorem \ref{co-Max} shows that $Z_G(H)/H$ is finite.
  For $\so (2,1)\subset \so (2,1)\times \so (2,1)$ the statement
  remains true if $\so (2,1)$ is embedded diagonally, but clearly not
  if $\so (2,1)\simeq \so (2,1)$ is one of the ideals.
\end{remark}

We now provide an elementary but useful description of the Lie algebra
structures of $\so(p,q+1)$ and $\so(p+1,q)$ in terms of
$\so(p,q)$-modules. In the next result $\so(p,q)\ltimes\R^{p,q}$ is
considered with the usual semidirect product Lie algebra structure
coming from the fact that $\R^{p,q}$ is an $\so(p,q)$-module. Also, we
will denote:
$$
I_{p,q}(c) = \begin{pmatrix} c & 0 \cr
0 & I_{p,q}\end{pmatrix}\text{ if } c>0 \quad\text{and}\quad
I_{p,q}(c) = \left(
  \begin{matrix}
    I_{p,q} & 0 \\
    0 & c
  \end{matrix}
\right) \text{ if } c<0\, .
$$

\begin{lemma}\label{lemma-Lie-so(p,q)Rpq}
  For $p,q \geq 1$, $n = p + q \geq 3$ and every $c\in \R$, let:
  $$
  [\cdot,\cdot]_c : \so(p,q)\oplus\R^{p,q}\times
  \so(p,q)\oplus\R^{p,q} \rightarrow \so(p,q)\oplus\R^{p,q},
  $$
  be given by:
  \begin{itemize}
  \item $[X,Y]_c = XY-YX$ for $X,Y\in \so(p,q)$.
  \item $[X,u]_c = -[u,X]_c = X(u)$ for $X \in \so(p,q)$ and $u \in
    \R^{p,q}$.
  \item $[u,v]_c = T_c(u\wedge v)$ for $u,v \in \R^{p,q}$, where $T_c$
    is the map defined in Lemma~\ref{lemma-Rpq-so(p,q)}.
  \end{itemize}
  Then, $[\cdot,\cdot]_c$ defines a Lie algebra structure on
  $\so(p,q)\oplus\R^{p,q}$ which satisfies:
  \begin{enumerate}
  \item $(\so(p,q)\oplus\R^{p,q},[\cdot,\cdot]_0) \simeq
    \so(p,q)\ltimes\R^{p,q}$.
  \item $(\so(p,q)\oplus\R^{p,q},[\cdot,\cdot]_c) \simeq
    \so(\R^{n+1},I_{p,q}(c))$ for every $c\neq 0$.
  \end{enumerate}
  In particular, $(\so(p,q)\oplus\R^{p,q},[\cdot,\cdot]_c)$ is
  isomorphic to $\so(p+1,q)$ $(\so(p,q+1))$ for $c > 0$ $(c < 0$,
  respectively$)$ under an isomorphism for which the summand
  $\so(p,q)$ is canonically embedded.
\end{lemma}
\begin{proof}
  Using Lemma~\ref{lemma-Rpq-so(p,q)} it is an easy exercise to prove
  that $[\cdot,\cdot]_c$ defines a Lie algebra structure. On the other
  hand, (1) is the definition of the Lie brackets on
  $\so(p,q)\ltimes\R^{p,q}$.

  Finally, for (2) an isomorphism is easily seen to be given by:
  \[(X,u)\mapsto \begin{pmatrix} 0 & u^* \cr cu & X\end{pmatrix}\,
  ,\,\, c>0, \text{ and } (X , u)\mapsto \begin{pmatrix}
    X & cu \\
    u^* & 0
  \end{pmatrix}, \,\, c<0
  \]
  where $u \in \R^{p,q}$ is considered as a column vector and $u^* =
  -u^t I_{p,q}$. For $c>0$ an isomorphism
  $\so(\R^{n+1},I_{p,q}(c))\simeq \so (p+1,q)$ is
  \[\begin{pmatrix} 0 & u^* \cr
    cu & X\end{pmatrix}\mapsto \begin{pmatrix} 0 & \sqrt{c} u^* \cr
    \sqrt{c} u & X
  \end{pmatrix}
  \]
  and similarly for $c<0$.
\end{proof}

Next, we state a uniqueness property for $\so(p,q)$-invariant inner
products related to the constructions of
Lemma~\ref{lemma-Lie-so(p,q)Rpq}. Its proof follows easily from
Schur's Lemma and the uniqueness (up to a multiple) of the Killing
form of complex simple Lie algebras.

\begin{lemma}\label{lemma-inner-so(p,q)Rpq}
  Assume that $p,q \geq 1$, $n = p+q\ge 3$ and $n \not= 4$. Let
  $\left<\cdot,\cdot\right>_1$ and $\left<\cdot,\cdot\right>_2$ be
  inner products on $\so(p,q)$ and $\R^{p,q}$, respectively. Assume
  that $\left<\cdot,\cdot\right>_1$ and $\left<\cdot,\cdot\right>_2$
  are $\so(p,q)$-invariant, in other words:
\begin{itemize}
\item $\left<[X,Y],Z\right>_1 = -\left<Y,[X,Z]\right>_1$ for every
  $X,Y,Z \in \so(p,q)$, and
\item $\left<X(u),v\right>_2 = -\left<u,X(v)\right>_2$ for every $X
  \in \so(p,q)$ and $u,v \in \R^{p,q}$,
\end{itemize}
If $c \in\R\setminus \{0\}$ is given, then there exist $a_1, a_2 \in
\R$ such that $a_1\left<\cdot,\cdot\right>_1 +
a_2\left<\cdot,\cdot\right>_2$ is the Killing form of
$(\so(p,q)\oplus\R^{p,q},[\cdot,\cdot]_c)$.
\end{lemma}

\begin{remark}
  The previous result is not valid for $n = p + q = 4$. For $p=q=2$,
  the invariant bilinear forms on $\so(2,2) \simeq \so(2,1) \times
  \so(2,1)$ are the linear combinations of the Killing forms of the
  factors.

  On the other hand, for the realification $\g_\R$ of a simple complex
  Lie algebra, the complex estructure $J$ induces the invariant
  bilinear form $\widehat{K}(X,Y) = K(X,J(Y))$, where $K$ is its
  Killing form considered as a real Lie algebra. Note that
  $\widehat{K}$ is not a multiple of $K$. In fact, using Schur's lemma
  one can show that the invariant bilinear forms on $\g_\R$ are given
  by $aK + b\widehat{K}$, where $a,b \in \R$.

  Furthermore, for $\mathfrak{u}$ a compact real form of $\g$ it is
  easy to see that a $\g_\R$-invariant bilinear map of the form
  $\left<\cdot,\cdot\right> = aK + b\widehat{K}$ has $b = 0$ if and
  only $\mathfrak{u}$ and $J\mathfrak{u}$ are perpendicular with
  respect respect to $\left<\cdot,\cdot\right>$. This follows from the
  fact that $\g_\R = \mathfrak{u} \oplus J\mathfrak{u}$ is a Cartan
  decomposition and the properties of such decompositions with respect
  to the Killing form. From this, it is easy to see that the proof of
  Lemma~\ref{lemma-inner-so(p,q)Rpq} remains valid for $\so(3,1)
  \simeq \gsl(2,\C)_\R$ if we further assume that $\mathfrak{u}$
  and $J\mathfrak{u}$ are perpendicular with respect to
  $\left<\cdot,\cdot\right>_1$ for some compact form $\mathfrak{u}$ of
  $\gsl(2,\C)$.
\end{remark}
\end{appendices}


\begin{thebibliography}{00}
\bibitem{Akivis} M.A.~Akivis and V.V.~Goldberg, Conformal differential
  geometry and its generalizations. Pure and Applied Mathematics (New
  York). A Wiley-Interscience Publication. John Wiley \& Sons, Inc.,
  New York, 1996.

\bibitem{Bader-Frances-Melnick} U.~Bader, C.~Frances and K.~Melnick, An
  embedding theorem for automorphism groups of Cartan geometries. GAFA
  19 (2009), no. 2, 333--355.

\bibitem{Bader-Nevo} U.~Bader and A.~Nevo, Conformal actions of simple
  Lie groups on compact pseudo-Riemannian manifolds.  J. Differential
  Geom.  {\bf 60} (2002), no. 3, 355--387.

\bibitem{BourbakiLie7-9} N.~Bourbaki, Lie groups and Lie
  algebras. Chapters 7--9. Elements of Mathematics
  (Berlin). Springer-Verlag, Berlin, 2005.

\bibitem{Cairns} G.~Cairns, G\'eom\'etrie globale des feuilletages
  totalement g\'eod\'esiques.  C. R. Acad. Sci. Paris S\'er. I Math.
  297 (1983), no. 9, 525--527.

\bibitem{Cairns-Ghys} G.~Cairns and E.~Ghys, Totally geodesic
  foliations on $4$-manifolds. J. Differential Geom.  23 (1986),
  no. 3, 241--254.

\bibitem{GCT} A.~Candel and R.~Quiroga--Barranco, Gromov's centralizer
  theorem, Geom.~Dedicata {\bf 100} (2003), 123--155.

\bibitem{D52} E.~B.~Dynkin, Maximal subgroups of the classical
  groups. (Russian) Trudy Moskov. Mat. Ob\v{s}\v{c}. {\bf 1},
  (1952), 39--166. See also: Translations of the AMS (2) {\bf 6},
  (1957), 245--378. 

\bibitem{Frances-Melnick} C.~Frances and K.~Melnick, Conformal actions
  of nilpotent groups on pseudo-Riemannian manifolds, Duke
  Math. J. {\bf 153} (2010), no. 3, 511--550.

\bibitem{Gromov} M.~Gromov, Rigid transformations groups, in
  G\'eom\'etrie diff\'erentielle, Colloque G\'eom\'etrie et Physique
  de 1986 en l'honneur de Andr\'e Lichnerowicz (D. Bernard and Y.
  Choquet-Bruhat, eds.), Hermann, 1988, 65--139.

\bibitem{GoodmanWallach} R.~Goodman and N.R.~Wallach, Representations
  and Invariants of the Classical Groups. Encyclopedia of Mathematics
  and its Applications, \textbf{68}. Cambridge University Press,
  Cambridge, 1998.

\bibitem{Helgason} S.~Helgason, Differential Geometry, Lie Groups and
  Symmetric Spaces, Pure and Applied Mathematics \textbf{80}, Academic
  Press Inc., 1978.

\bibitem{ONeill-book} B.~O'Neill, Semi-Riemannian Geometry. With
  Applications to Relativity. Pure and Applied Mathematics,
  \textbf{103}.  Academic Press, Inc., New York, 1983.

\bibitem{Onishchik} A.~L.~Onishchik, Lectures on Real Semisimple Lie
  Algebras and their Representations, ESI Lectures in Mathematics and
  Physics. European Mathematical Society (EMS), Z\"urich, 2004.

\bibitem{Quiroga-Annals} R.~Quiroga-Barranco, Isometric actions of
  simple Lie groups on pseudo-Riemannian manifolds, Ann. of Math. (2)
  {\bf 164} (2006), no. 6, 941--969.

\bibitem{Quiroga-Z} R.~Quiroga-Barranco, Isometric actions of simple
  groups and transverse structures: the integrable normal case,
  ``Geometry, Rigidity and Group Actions'', 229--261, Chicago Lectures
  in Mathematics Series, The University of Chicago Press.

\bibitem{Szaro} J.~Szaro, Isotropy of semisimple group actions on
  manifolds with geometric structure, Amer. J. Math. {\bf 120} (1998),
  129--158.

\bibitem{Warner72} G. Warner, Harmonic Analysis on Semi-Simple Lie
  Groups I, Springer, Berlin and New York, 1972.

\bibitem{Zeghib} A.~Zeghib, On affine actions of Lie groups,
  Math. Z. {\bf 227} (1998), no. 2, 245--262.

\bibitem{Zimmer-Lorentz} R.~J.~Zimmer, On the automorphism group of a
  compact Lorentz manifold and other geometric manifolds.
  Invent. Math. {\bf 83} (1986), no. 3, 411--424.

\bibitem{Zimmer-reps} R.~J.~Zimmer, Representations of fundamental
  groups of manifolds with a semisimple transformation group.
  J. Amer. Math. Soc.  {\bf 2} (1989), no. 2, 201--213.

\bibitem{Zimmer-rigid} R.~J.~Zimmer, Automorphism groups and
  fundamental groups of geometric manifolds.  Differential geometry:
  Riemannian geometry (Los Angeles, CA, 1990), 693--710,
  Proc. Sympos. Pure Math., 54, Part 3, Amer. Math. Soc., Providence,
  RI, 1993.

\bibitem{Zimmer-entropy} R.~J.~Zimmer, Entropy and arithmetic
  quotients for simple automorphism groups of geometric manifolds,
  Geom. Dedicata {\bf 107} (2004), 47--56.


\end{thebibliography}
\end{document}